\documentclass[11pt,reqno]{amsart}

\usepackage{amsmath, amsfonts, amsthm, amssymb, stmaryrd, color, cite}

\usepackage{latexsym,hyperref}

\newtheorem{theorem}{Theorem}[section]
\newtheorem{lemma}{Lemma}[section]
\newtheorem{proposition}{Proposition}[section]

\theoremstyle{definition}
\newtheorem{definition}{Definition}[section]

\theoremstyle{remark}

\numberwithin{equation}{section}
\allowdisplaybreaks

\def\f{\frac}

\def\hf1{^\f{1}{1-\xi^2}}

\def\be{\begin{equation}}
\def\en{\end{equation}}
\def\bs{\begin{split}}
\def\es{\end{split}}
\def\ba{\begin{align}}
\def\ea{\end{align}}

\author[Z. Qiu]{Zhaoyang Qiu}
\address{School of Mathematics and Statistics, Huazhong University of Science and Technology, Wuhan, 430074, China.}
\email{tangyb@hust.edu.cn. zhqmath@163.com}

\author[Y. Tang]{Yanbin Tang}

\title[strong pathwise solution and LDP to the stochastic Boussinesq equations]
{Strong pathwise solution and large deviation principle for the stochastic Boussinesq equations with partial diffusion term}

\keywords{Stochastic Boussinesq equations, strong pathwise solution, large deviation principle, global existence, stochastic compactness.}
\subjclass[2010]{35Q35, 76D05, 60H15, 60F10}

\date{\today}

\begin{document}
\begin{abstract}
We establish the existence and uniqueness of local strong pathwise solutions to the stochastic Boussinesq equations with partial diffusion term forced by multiplicative noise on the torus in $\mathbb{R}^{d},d=2,3$. The solution is strong in both PDE and probabilistic sense. In the two dimensional case, we prove the global well-posedness of strong pathwise solutions to the system  with large initial data forced by additive noise. After the global existence and uniqueness of strong solutions are established, the large deviation principle (LDP) is proved by the weak convergence method.  The weak convergence is shown by a tightness argument in the appropriate functional space.
\end{abstract}

\maketitle
\section{Introduction}

The Boussinesq equations are widely considered as the fundamental model for the study of large scale atmospheric and oceanic flows, built environment, dispersion of dense gases and internal dynamical structure of stars, which also retain some features of Navier-Stokes and Euler equations, see \cite{Busse,Gibbon,Pedlosky} for further background. The PDEs perturbed randomly are considered as a primary tool in the modeling of uncertainty, especially while describing fundamental phenomenon in physics, climate dynamics, communication systems and gene regulation systems. Hence, the study of the well-posedness and dynamical behaviour of PDEs subjected to the noise which is largely applied to the theoretical and practical areas has drawn a lot of attention. In this paper we consider the following stochastic Boussinesq equations with partial diffusion term driven by multiplicative noise:
\begin{eqnarray}\label{Equ1.1}
\left\{\begin{array}{ll}
du-\Delta udt+(u\cdot \nabla)udt+\nabla \pi dt=\theta e_{d}dt+f(u,\theta)d\mathcal{W},\\
d\theta+(u\cdot\nabla)\theta dt=0,\\
\nabla\cdot u=0,\\
\end{array}\right.
\end{eqnarray}
where  $u=(u_{1},\cdots,u_{d}), d=2,3$,  $\pi$ and $\theta$ denote the velocity, pressure and temperature, respectively;
$e_{2}=(0,1),e_{3}=(0,0,1)$, $\mathcal{W}$ is a Q-Wiener process that will be introduced in Section 2. The initial conditions are random variables $u(0,\cdot)=u_{0}(\omega, x),~~ \theta(0,\cdot)=\theta_{0}(\omega, x)$ with sufficient spatial regularity introduced later. We focus on the periodic boundary conditions, with the spatial domain being the torus $\mathbb{T}^{d}=(-\pi,\pi)^{d}, d=2,3$. In short, the Boussinesq equations model the interaction between the incompressible fluid flow and thermal dynamics. We will refer to the first and second equations in system \eqref{Equ1.1} as equations of momentum and temperature, respectively. For the deterministic Boussinesq equations, there have been some existence and regularity results in \cite{Yzhuan,Ye,Zhuan} for the full viscous case and in \cite{Chae} for the partial viscous case. The well-posedness of solutions in two dimensions with noise driven by the cylindrical Wiener process was given in \cite{Trenchea} for strong solution, in \cite{Chueshov} for global solutions that are weak in PDE sense and strong in probability sense, in \cite{Brzezniak, Yamazaki} for martingale solutions, and in \cite{Alonso, du} for maximal strong solutions, especially \cite{du} established the global existence of strong solution where system was driven by non-degenerate multiplicative noises which provided a regularizing effect, the initial data was also required to be sufficiently small.
For examples of results on the well-posedness of solutions to the system driven by other types of noise such as fractional Brownian motion or L\'evy noise we refer the reader to \cite{Bessaih, Huang1, Huang2, HuangLi}. If $\theta=0$, the system \eqref{Equ1.1} reduces to the stochastic Navier-Stokes equations, for which substantial developments have been made in recent years, see for example, \cite{Temam,Millet,Breit,Gibbon,Flandoli2,Glatt,Chueshov,Rozovskii,Masmoudi,Zhang,Sundar,DWang}. Also for the stochastic Euler equations, we refer the reader to \cite{BessaihH,Flandoli1,PeszatS,Mariani,Kim1}.
Here we shall study the well-posedness and large deviation principle of strong solutions to the Boussinesq system \eqref{Equ1.1} with only partial diffusion, that is, only the equation of momentum has viscous diffusion but the equation of temperature has no diffusion.

 We now give an overview of this paper, including the main difficulties and our new ideas and results.  For the first part, we mainly concentrate on proving the local existence and uniqueness of strong solution to the system \eqref{Equ1.1} forced by multiplicative noise in dimensions two and three, where  the solutions   are strong in both PDE and probabilistic sense evolving continuously in $H^{s}$ for integer $s>\frac{d}{2}+1$. A key ingredient which allows us to obtain the strong pathwise solution is to establish the compactness in suitable functional space. In the stochastic setting, the embedding $L^{2}(\Omega;X)\hookrightarrow L^{2}(\Omega;Y)$ might not be compact, even if $X\hookrightarrow\hookrightarrow Y$. As a result, the usual compactness criteria, such as the Aubin or Arzel\`{a}-Ascoli type theorems, can not be used directly. Thus, we rely on the Yamada-Watanabe type argument to obtain the pathwise solution after we establish the existence of martingale solution and pathwise uniqueness.

  The first difficulty we meet in the proof of above mentioned local existence is how to construct a suitable approximation scheme. Actually, the term $\|\nabla u,\nabla \theta\|_{L^{\infty}}\cdot\|u,\theta\|_{s}$ appears when we establish a priori estimates for approximation solutions, which prevents us from closing the a priori $L^{2}(\Omega; H^{s})$ estimate. Inspired by \cite{Kim1}, we add a cut-off function to render the nonlinear term, which will allow us to obtain the uniform a priori estimates. By a fully exploit the hidden structure, we find that $\|\nabla \theta\|_{L^{\infty}}$ can be controlled by the initial data and $\|\nabla u\|_{L^{\infty}}$, so it is enough for the cut-off function which depends only on $\|\nabla u\|_{L^{\infty}}$. This would allow us to adapt a mixed method to construct the approximation solutions, that is, the equation of temperature is solved directly by applying the standard method of characteristics, while the momentum equation is approximated by a finite dimensional Galerkin scheme. However, the cut-off function brings difficulty in the proof of uniqueness which plays a crucial role in the process of passing martingale solution to strong pathwise solution. In order to overcome this difficulty, we first show that there exists a pathwise solution of high regularity in $H^{s'}$ for integer $s'=s+1$.

  With the spirit of \cite{Lai} and \cite{Glatt-Holtz}, we can extend the range of regularity index to $s>\frac{d}{2}+1$ by applying a density and stability argument. Here, the term $\|\eta\|_{{s-1}}^{2}(\|u_{j}\|_{s+1}^{2}+\|\theta_{j}\|_{s+1}^{2})$ arises due to the coupled construction of system (\ref{Equ1.1}) with $\eta=\theta_{j}-\theta_{k}$, and  the desired convergence requires  $\mathbb{E}\|\eta\|_{{s-1}}^{2}(\|u_{j}\|_{s+1}^{2}+\|\theta_{j}\|_{s+1}^{2})\sim o(1)$ which makes our estimates more complicated.

The main contribution of this paper is to obtain the global existence of strong pathwise solution of the system \eqref{Equ1.1} with large initial data forced by additive noise in dimension two. We mention that the diffusion term $\Delta u$ plays an essential role throughout the proof. Here, we use a stochastic analogue of logarithmic Gronwall's lemma to show that $\sup_{t\in [0,\xi\wedge T]}\|\nabla w, \nabla\theta\|_{L^{4}}<\infty$, where $w={\rm curl}u$ and $\xi$ is the maximal existence time of the strong pathwise solution. However, unlike in \cite{Glatt-Holtz}, this estimate is not sufficient for our case to conclude that
$\xi=\infty$ a.s. since  $\xi$ might not be a blow-up time under this norm. To overcome the difficulty, we establish a non-blowup criterion for the solution in the stochastic case, which shows that once we built the bound of the gradient of the solution in $L^{\infty}$ that can be controlled by $\|\nabla w\|_{L^{4}}$ and $\|\nabla u\|_{L^{2}}$ using the Sobolev embedding and the Biot-Savart law, the global result will follow. We point out that the coupled constitution makes the proof of the bound  $\sup_{t\in [0,\xi\wedge T]}\|\nabla w, \nabla\theta\|_{L^{4}}$ more complex compared to the Euler equation \cite{Glatt-Holtz}, more technique estimates are required, see Proposition 5.2.

After achieving the global well-posedness of solutions to the system \eqref{Equ1.1},  we turn to proving the large deviation principle by the weak convergence approach based on the variational representations of infinite-dimensional Wiener processes introduced by \cite{Budhiraja, Dupuis}. Authors in \cite{Chueshov} and \cite{Duan} also achieved large deviations in space $\mathcal{C}([0,T];H) \cap L^{2}(0,T;V)$ for the two dimensional stochastic Boussinesq equations by weak convergence approach whereas we establish it in space $\mathcal{X}=[\mathcal{C}([0,T];X^{s-1})\cap L^{2}(0,T;X^{s})]\times \mathcal{C}([0,T];H^{s-1})$ for any integer $s>2$. Unlike the system considered here, their equation of temperature has a diffusion term which proves to be useful in their estimates. Actually, the space $\mathcal{X}$ is a ``nonoptimal'' space. This is due to the fact that no diffusion term appears in the temperature equation, and hence  the weak convergence will be proved using the tightness argument by following ideas from \cite{Millet}.

The paper is organized as follows. In Section 2, we recall some deterministic and stochastic preliminaries associated with system \eqref{Equ1.1} and then state our results. In Section 3, we build the existence of strong pathwise solution in $H^{s}$ for integer $s>\frac{d}{2}+2$. Then, we extend the existence of strong pathwise solution to $H^s$ for integer $s>\frac{d}{2}+1$ using a density and stability argument in Section 4. Section 5 is devoted to applying the stochastic logarithmic Gronwall  lemma to establish the global existence of strong pathwise solution under the case of additive noise in two dimensional. The large deviation principle is then proved in Section 6.
\bigskip

\section{Preliminaries and main results}\label{sec2}
\setcounter{equation}{0}
In this section, we begin by reviewing some deterministic and stochastic preliminaries associated with system (1.1) and then give our main results.

For each integer $s\geq0$, let
\begin{equation}\label{2.1}
X^{s}(\mathbb{T}^{d})=\left\{u\in (H^{s}(\mathbb{T}^{d}))^{d}:\nabla\cdot u=0\right\},
\end{equation}
where
\begin{eqnarray*}
H^{s}(\mathbb{T}^{d}):=\left\{u\in L^{2}(\mathbb{T}^{d}):\|u\|_{H^{s}}^{2}=\sum_{k\in \mathbb{Z}}(1+k^{2})^{s}|\hat{u}(k)|^{2}<\infty\right\},
\end{eqnarray*}
denotes the Sobolev spaces of functions having distributional derivatives up to order $s\in \mathbb{N}^{+}$ integrable in $L^{2}(\mathbb{T}^{d})$. The inner product of $X^{s}$ will be denoted by $(\cdot,\cdot)_{X^{s}}=\sum_{|\alpha|\leq s}(\partial^{\alpha}\cdot,\partial^{\alpha}\cdot)_{L^{2}}$.
We denote by $P$ the Leray projector which is the orthogonal projection from $L^{2}(\mathbb{T}^{d})$ into the closed subspace $L^2_{\rm div}(\mathbb{T}^d):=\{u\in (L^{2}(\mathbb{T}^{d}))^{d}:\nabla\cdot u=0\}$ and is also a bounded linear operator from $H^{s}(\mathbb{T}^{d})$ into $X^{s}(\mathbb{T}^{d})$.

In order to estimate the nonlinear terms, we shall use the following commutator and Moser estimates which were proved in \cite{Kato2,MajdaA}.
\begin{lemma}\label{lem2.1}
For $u,v\in H^{s}(\mathbb{T}^{d})$ if $s>\frac{d}{2}+1$,
\begin{equation}\label{2.2}
\sum\limits_{0\leq|\alpha|\leq s}\|\partial^{\alpha}(u\cdot \nabla) v-u\cdot \nabla \partial^{\alpha}v\|_{L^{2}}\leq C(\|\nabla u\|_{L^{\infty}}\|v\|_{s}+\|\nabla v\|_{L^{\infty}}\|u\|_{s}),
\end{equation}
and if $s>\frac{d}{2}$,
 \begin{equation}\label{2.3}
\|uv\|_{s}\leq C(\|u\|_{L^{\infty}}\|v\|_{s}+\|v\|_{L^{\infty}}\|u\|_{s}),
\end{equation}
for some positive constants $C=C(s,\mathbb{T}^{d})$ independent of $u$ and $v$.
\end{lemma}

By Lemma \ref{lem2.1}, we can use the H\"{o}lder inequality to obtain the following estimates which will be applied throughout the rest of the paper.

\begin{lemma}\label{lem2.2}
For $s>\frac{d}{2}+1$, there exists a positive constant $C=C(s,\mathbb{T}^{d})$ such that,

 (i) if $u\in H^{s}(\mathbb{T}^{d})$ and $v\in H^{s+1}(\mathbb{T}^{d})$, then $P(u\cdot\nabla v)\in X^{s}(\mathbb{T}^{d})$ and
\begin{equation*}
\|P(u\cdot\nabla) v\|_{s}\leq C(\|u\|_{L^{\infty}}\|v\|_{s+1}+\|\nabla v\|_{L^{\infty}}\|u\|_{s}),
\end{equation*}

(ii)  if $u\in H^{s}(\mathbb{T}^{d})$ and $v\in H^{s+1}(\mathbb{T}^{d})$, then
\begin{equation*}
\left|\sum\limits_{|\alpha|\leq s}(\partial^{\alpha}(u\cdot \nabla) v,\partial^{\alpha}v)\right|\leq C(\|\nabla v\|_{L^{\infty}}\|u\|_{s}+\|\nabla u\|_{L^{\infty}}\|v\|_{s})\|v\|_{s}.
\end{equation*}
\end{lemma}

The following are some spaces used for solutions involving fractional derivative in time. These spaces are useful since solutions of stochastic evolution system are H\"{o}lder continuous of order strictly less than $\frac{1}{2}$ with respect to time.

For any fixed $p>1$ and $\alpha\in(0,1)$ we define,
\begin{equation*}
W^{\alpha,p}(0,T;X):=\left\{v\in L^{p}(0,T;X):\int_{0}^{T}\int_{0}^{T}\frac{\|v(t_{1})-v(t_{2})\|_{X}^{p}}{|t_{1}-t_{2}|^{1+\alpha p}}dt_{1}dt_{2}<\infty\right\},
\end{equation*}
endowed with the norm,
\begin{equation*}
\|v\|^p_{W^{\alpha,p}(0,T;X)}:=\int_{0}^{T}\|v(t)\|_{X}^{p}dt+\int_{0}^{T}\int_{0}^{T}\frac{\|v(t_{1})-v(t_{2})\|_{X}^{p}}{|t_{1}-t_{2}|^{1+\alpha p}}dt_{1}dt_{2},
\end{equation*}
where $X$ is a separable Hilbert space.
For the case $\alpha=1$, we take,
\begin{equation*}
W^{1,p}(0,T;X):=\left\{v\in L^{p}(0,T;X):\frac{dv}{dt}\in L^{p}(0,T;X)\right\},
\end{equation*}
which is the classical Sobolev space with its usual norm,
\begin{equation*}
\|v\|_{W^{1,p}(0,T;X)}^{p}:=\int_{0}^{T}\|v(t)\|_{X}^{p}+\left\|\frac{dv}{dt}(t)\right\|_{X}^{p}dt.
\end{equation*}
Note that for $\alpha\in(0,1)$, $ W^{1,p}(0,T;X)$ is a subspace of $ W^{\alpha,p}(0,T;X)$.

We will use the following compact embedding results given in  \cite{Flandoli2} to achieve a tightness argument.
\begin{lemma}\!\!\!{\rm\cite[Theorems 2.1]{Flandoli2}}\label{lem2.3} Suppose that $X_{1}\subset X_{0}\subset X_{2}$ are Banach spaces and $X_{1}$ and $X_{2}$ are reflexive and the embedding of $X_{1}$ into $X_{0}$ is compact.
Then for any $1<p<\infty,~ 0<\alpha<1$, the embedding,
\begin{eqnarray*}
L^{p}(0,T;X_{1})\cap W^{\alpha,p}(0,T;X_{2})\hookrightarrow L^{p}(0,T;X_{0}),
\end{eqnarray*}
and
\begin{eqnarray*}
\mathcal{C}([0,T];X_1)\cap \mathcal{C}^\alpha([0,T];X_2)\hookrightarrow \mathcal{C}([0,T];X_0),
\end{eqnarray*}
is compact.
\end{lemma}

We now describe the stochastic setting of the problem. Let $\mathcal{S}:=(\Omega,\mathcal{F},\{\mathcal{F}_{t}\}_{t\geq0},\mathbb{P}, \mathcal{W})$ be a fixed stochastic basis and $(\Omega,\mathcal{F},\mathbb{P})$ a complete probability space. Denote $Q$ as a linear positive, trace class (hence compact) operator in Hilbert space $H$ and let $\mathcal{W}$ be a Wiener process defined on the Hilbert space $H$ with covariance operator $Q$, which is adapted to the complete, right continuous filtration $\{\mathcal{F}_{t}\}_{t\geq 0}$. If $\{e_{k}\}_{k\geq 1}$ is a complete orthonormal basis of $H$ such that $Qe_{i}=\lambda_{i}e_{i}$, then $\mathcal{W}$ can be written formally as the expansion $\mathcal{W}(t,\omega)=\sum\limits_{k\geq 1}\sqrt{\lambda_{k}}e_{k}\mathcal{W}_{k}(t,\omega)$ where $\{\mathcal{W}_{k}\}_{k\geq 1}$ is a sequence of independent standard one-dimensional Brownian motions. We also have that $\mathcal{W}\in \mathcal{C}([0,\infty),H)$ a.s. see \cite{Prato}.

Let $H_{0}=Q^{\frac{1}{2}}H$, then $H_{0}$ is a Hilbert space with the inner product
\begin{equation*}
\langle h,g\rangle_{H_{0}}=\langle Q^{-\frac{1}{2}}h,Q^{-\frac{1}{2}}g\rangle_{H},~ \forall~~ h,g\in H_{0},
\end{equation*}
with the induced norm $\|\cdot\|_{H_{0}}^{2}=\langle\cdot,\cdot\rangle_{H_{0}}$. The imbedding map $i:H_{0}\rightarrow H$ is Hilbert-Schmidt and hence compact operator with $ii^{\ast}=Q$. Now consider another separable Hilbert space $X$ and let $L_{Q}(H_{0},X)$ be the space of linear operators $S:H_{0}\rightarrow X$ such that $SQ^{\frac{1}{2}}$ is a linear Hilbert-Schmidt
operator from $H$ to $X$,  endowed with
  the norm $\|S\|_{L_{Q}}^{2}=tr(SQS^{\ast}) =\sum\limits_{k}| SQ^{\frac{1}{2}}e_{k}|_{X}^{2}$.
 Set $L_2(H,X):=\{SQ^{\frac{1}{2}}: \, S\in L_{Q}(H_{0},X)\}$ with the norm defined by
  $\|f\|^2_{L_2(H,X)}  =\sum\limits_{k}|fQ^{\frac{1}{2}}e_{k}|_{X}^{2}$.  
%

For a $X$-valued predictable process $G\in L^{2}(\Omega;L^{2}_{loc}([0,\infty),L_{2}(H,X)))$  by taking $G_{k}=GQ^{\frac{1}{2}}e_{k}$, one can define the stochastic integral,
\begin{equation*}
M_{t}:=\int_{0}^{t}Gd\mathcal{W}=\sum_{k}\int_{0}^{t}
GQ^{\frac{1}{2}}e_{k}d\mathcal{W}_{k}=\sum_{k}\int_{0}^{t}G_{k}d\mathcal{W}_{k},
\end{equation*}
which is an $X$-valued square integrable martingales, and the Burkholder-Davis-Gundy inequality holds,
\begin{equation}\label{2.4}
\mathbb{E}\left(\sup_{0\leq t\leq T}\left\|\int_{0}^{t}Gd\mathcal{W}\right\|_{X}^{p}\right)\leq c_{p}\mathbb{E}\left(\int_{0}^{T}\|G\|_{L_{2}(H,X)}^{2}dt\right)^{\frac{p}{2}},
\end{equation}
for any $p\geq1$, for more details see \cite{Prato}.

 We shall also use the stochastic integrals evolving on $W^{s,p}(\mathbb{T}^{d})$, and recall some details of the construction given in \cite{Rozovskii1}. Suppose that $p\geq 2, m\geq 0$ and define
\begin{equation*}
\mathbb{W}^{s,p}:=\left\{f:\mathbb{T}^{d}\rightarrow L_{2}~~ {\rm such~ that}~~ f_{k}(\cdot)=f(\cdot)e_{k}\in W^{s,p}, \sum_{|\alpha|\leq m}\int_{\mathbb{T}^{d}}|\partial^{\alpha}f|_{L_{2}}^{p}dx<\infty\right\},
\end{equation*}
which is a Banach space with norm,
\begin{equation*}
\|f\|_{\mathbb{W}^{s,p}}^{p}:=\sum_{|\alpha|\leq s}\int_{\mathbb{T}^{d}}|\partial^{\alpha}f|_{L_{2}}^{p}dx.
\end{equation*}

We next introduce the conditions imposed on the noise intensity $f$. For Banach spaces $X$ and $Y$, let $\mathcal{L}(X,Y)$ be the space of functions $f\in \mathcal{C}(X\times [0,\infty);Y)$ that satisfy the linear growth and Lipschitz conditions. Namely, there exist positive constants $C_{1}$, $C_{2}$ such that,
\begin{eqnarray*}
&&\|f(x,t)\|_{Y} \leq C_{1}(1+\|x\|_{X}), \hspace{.5cm} \text{for all } x\in X, t>0,\label{growth}\\
&&\|f(x,t)-f(y,t)\|_{Y} \leq C_{2}\|x-y\|_{X}, \hspace{.5cm} \text{for all } x,y\in X, t>0.
\end{eqnarray*}
Denote the space of functions $f\in \mathcal{C}(X\times [0,\infty);Y)$ that only satisfy the linear growth condition as $\mathcal{L}_{g}(X,Y)$. Then, we assume that
\begin{equation}\label{2.7}
f\in {\rm \mathcal{L}}(L^{2},L_{2}(H;L^{2}))\cap {\rm \mathcal{L}}(H^{s-1},L_{2}(H;H^{s-1}))\cap {\rm \mathcal{L}}(H^{s},L_{2}(H;H^{s})),
\end{equation}
for fixed integer $s>\frac{d}{2}+1$. In the process of proving local existence of pathwise solution, we also impose additional conditions as follows,
\begin{eqnarray}
f\in \mathcal{L}_{g}(H^{s'},L_{2}(H;H^{s'})), \label{2.9}
\end{eqnarray}
 where $s'=s+1$. In the case of additive noise, we assume that
\begin{equation}\label{2.10}
f\in L^{2}(\Omega,L^{2}_{loc}([0,\infty);L_{2}(H;H^{s})))\cap L^{4}(\Omega,L^{4}_{loc}([0,\infty);\mathbb{W}^{2,4})),
\end{equation}
and that $f$ is predictable. Next, we introduce the definition of the local, maximal and global solutions to system (\ref{Equ1.1}).

\begin{definition}\label{def2.1}\rm{(Local strong pathwise solution)} Let $(\Omega,\mathcal{F},\{\mathcal{F}_{t}\}_{t\geq0},\mathbb{P},\mathcal{W})$ be a fixed probability space. Assume that $(u_{0},\theta_{0})$ is an $X^{s}\times H^{s}$-valued $\mathcal{F}_{0}$-measurable random variable and $f$ satisfies conditions (\ref{2.7})-(\ref{2.9}).

 (i) The triple $(u,\theta,\tau)$ is a local strong pathwise solution if $\tau$ is a strictly positive stopping time such that $(u(\cdot \wedge \tau),\theta(\cdot \wedge \tau))$ is an $X^{s}\times H^{s}$-valued $\{\mathcal{F}_{t}\}$-progressively measurable process for $s>\frac{d}{2}+1,~d=2,3$, with,
\begin{equation*}
u(\cdot \wedge \tau)\in \mathcal{C}([0,\infty),X^{s})\cap L_{loc}^{2}(0,\infty;X^{s+1}),\hspace{.3cm} \text{and} \hspace{.3cm} \theta(\cdot \wedge \tau)\in \mathcal{C}([0,\infty),H^{s}), ~\mathbb{P}~\mbox{a.s.}
\end{equation*}
 and satisfying,
\begin{eqnarray}\label{2.11}
\left\{\begin{array}{ll}
u(t\wedge \tau)+\int_{0}^{t\wedge \tau}Au+P(u\cdot\nabla )udr=u_{0}+\int_{0}^{t\wedge \tau}P\theta e_{d}dr+\int_{0}^{t\wedge \tau}Pf(u,\theta)d\mathcal{W},\\
\theta(t\wedge \tau)+\int_{0}^{t\wedge \tau}(u\cdot\nabla )\theta dr=\theta_{0},
\end{array}\right.
\end{eqnarray}
where $A=-P\triangle u$ is the Stokes operator, for every $t\geq0$.

(ii) The pathwise uniqueness of the solution holds in the following sense:
if $(u_{1},\theta_{1},\tau_{1})$ and $(u_{2},\theta_{2},\tau_{2})$ are local strong pathwise solutions of system (1.1), with  $\mathbb{P}\{(u_{1}(0),\theta_{1}(0))=(u_{2}(0),\theta_{2}(0))\}=1$, then,
\begin{eqnarray*}
\mathbb{P}\{(u_{1}(t,x),\theta_{1}(t,x))=(u_{2}(t,x),\theta_{2}(t,x));\forall t\in[0,\tau_{1}\wedge\tau_{2}]\}=1.
\end{eqnarray*}
\end{definition}
\begin{definition}\label{def2.2} \rm{(Maximal and global solution)} A maximal pathwise solution is a triple $(u,\theta,\{\tau_{n}\}_{n\geq1}, \xi)$ such that each pair $(u,\theta,\tau_{n})$ is a local pathwise solution and $\{\tau_{n}\}$ is an increasing sequence with
$\lim_{n\rightarrow\infty}\tau_{n}=\xi$ and
\begin{equation*}
\sup\limits_{t\in[0,\tau_{n}]}\|\nabla u(t)\|_{L^{\infty}}\geq R,~~\mbox{on the set} ~~ \{\xi<\infty\}.
\end{equation*}
A maximal pathwise solution $(u,\theta,\{\tau_{n}\}_{n\geq1},\xi)$ is global if $\xi=\infty$ a.s.
\end{definition}

We now state the existence results of this paper.
\begin{theorem}\label{the2.1} Let $(\Omega,\mathcal{F},\{\mathcal{F}_{t}\}_{t\geq0},\mathbb{P},\mathcal{W})$ be a fixed probability space. Assume that $(u_{0},\theta_{0})$ is an $X^{s}\times H^{s}$-valued $\mathcal{F}_{0}$-measurable random variable for integer $s>\frac{d}{2}+1,~d=2,3$ and $f$ satisfies conditions (\ref{2.7})-(\ref{2.9}). Then there exists a unique maximal strong pathwise solution $(u,\theta,\{\tau_{n}\}_{n\geq1},\xi)$ of \eqref{Equ1.1} in the sense of Definitions 2.1 and 2.2.
\end{theorem}
\begin{theorem}\label{the2.2} Let $(\Omega,\mathcal{F},\{\mathcal{F}_{t}\}_{t\geq0},\mathbb{P}, \mathcal{W})$ be a fixed probability space. Assume that the noise is additive, and $f$ satisfies condition (\ref{2.10}). Then there exists a unique global strong pathwise solution of \eqref{Equ1.1} with $d=2$ in the sense of Definition 2.2.
\end{theorem}
For the following large deviation result, we consider the system below in 2D involving additive noise,
\begin{eqnarray}\label{2.12}
\left\{\begin{array}{ll}
du^{\epsilon}+Au^{\epsilon}dt+P(u^{\epsilon}\cdot\nabla)u^{\epsilon}dt=P\theta^{\epsilon} e_{2}dt+\sqrt{\epsilon}Pf d\mathcal{W},\\
d\theta^{\epsilon}+(u^{\epsilon}\cdot\nabla)\theta^{\epsilon}dt=0.
\end{array}\right.
\end{eqnarray}

\begin{theorem}\label{the2.3} Suppose that the initial data $(u_{0},\theta_{0})$ is an $X^{s}\times H^{s}$-valued $\mathcal{F}_{0}$-measurable random variable with integer $s>2$, and
\begin{eqnarray}\label{2.13}
f\in L^{p}(\Omega,\mathcal{C}([0,T];L_{Q}(H_{0};H^{s})))\cap L^{4}(\Omega,\mathcal{C}([0,T];\mathbb{W}^{2,4})), ~{\rm for} ~p\geq2.
\end{eqnarray}
Then, the solution $(u^{\epsilon},\theta^{\epsilon})$, $\epsilon\in (0,1]$ to system (\ref{2.12}) satisfies the large deviation principle in $\mathcal{X}$ with good rate function
\begin{eqnarray*}
I(U)=\inf_{\{h\in L^{2}(0,T;H_{0}):U=\mathcal{G}^{0}(\int_{0}^{\cdot}h(t)dt)\}}\left\{\frac{1}{2}\int_{0}^{T}\|h\|_{0}^{2}dt\right\}.
\end{eqnarray*}
\end{theorem}
We have reserved the details on the notation used above for Section 6.

\bigskip

\section{The existence of pathwise solution in $H^{s'}$}\label{sec3}
\setcounter{equation}{0}

In this section, we shall establish the existence and uniqueness of strong pathwise solution in $H^{s+1}$ given in four subsections. Due to the fact that only partial diffusion appears in system \eqref{Equ1.1}, to obtain uniform a priori estimates, in the spirit of \cite{Kim1}, we multiply the nonlinear terms by a smooth cut-off function depending on the size of $\|\nabla u\|_{L^{\infty}}$. Specifically, we first consider the trunction system of the form:
\begin{eqnarray}\label{3.1}
\left\{\begin{array}{ll}
du+Au dt+\varphi_{R}(\|\nabla u\|_{L^{\infty}})P(u\cdot\nabla)udt=P\theta e_{d}dt+Pf(u,\theta)d\mathcal{W},\\
d\theta+\varphi_{R}(\|\nabla u\|_{L^{\infty}})(u\cdot\nabla )\theta dt=0,\\
u(0,x)=u_{0}(x),~~~\theta(0,x)=\theta_{0}(x),
\end{array}\right.
\end{eqnarray}
where $\varphi_{R}:[0,\infty)\rightarrow [0,1]$ is $C^{\infty}$-smooth function defined as follows:
\begin{eqnarray*}
 \varphi_{R}(x) =\left\{\begin{array}{ll}
                  1,& \mbox{if} \ 0<x<R,  \\
                  0,& \mbox{if} \ x>2R.  \\
                \end{array}\right.
\end{eqnarray*}
Applying the operator $\nabla$ to both sides of the transport equation $(\ref{3.1})_2$, taking the inner product with $\nabla\theta|\nabla\theta|^{p-2}$ and integrating from $0$ to $t$, then passing $p\rightarrow \infty$, we obtain,
\begin{equation*}
\|\nabla \theta\|_{L^{\infty}}\leq \|\nabla \theta_{0}\|_{L^{\infty}}{\rm exp}\left(\int_{0}^{T}\varphi_{R}(\|\nabla u\|_{L^{\infty}})\|\nabla u\|_{L^{\infty}}dt\right),
\end{equation*}
for $t\in[0,T]$. Assume that $\|\nabla \theta_{0}\|_{L^{\infty}}\leq R$ a.s. then we have
\begin{equation}\label{3.2}
\|\nabla \theta\|_{L^{\infty}}\leq \|\nabla \theta_{0}\|_{L^{\infty}}{\rm exp}(2RT)\leq R{\rm exp}(2RT),
\end{equation}
where the constant $R$ is the same as in $\varphi_{R}$.

Here, we give the definition of martingale solution to system (\ref{3.1}).
\begin{definition}\label{adef2.1}\rm{(martingale solution)} Let $\Gamma$ be a Borel probability measure on $H^{s'}(\mathbb{T}^{3})\times X^{s'}(\mathbb{T}^{3})$ for integer $s'=s+1$. We say $\{(\Omega, \mathcal{F}, \{\mathcal{F}_{t}\}_{t\geq 0}, \mathbb{P}),u,\theta,\mathcal{W}\}$ is a global martingale solution to system (\ref{3.1}) with initial data law $\Gamma$ if the following conditions hold:\\
(i) $(\Omega, \mathcal{F}, \{\mathcal{F}_{t}\}_{t\geq 0}, \mathbb{P})$ is a stochastic basis and $\mathcal{W}$ is a Wiener process,\\
(ii) the processes
\begin{eqnarray*}
&&\theta(\cdot)\in L^{p}(\Omega;\mathcal{C}([0,T];H^{s'}),\\
&&u(\cdot)\in L^{p}(\Omega; \mathcal{C}([0,T];X^{s'}))\cap L^{p}(\Omega;L^{2}(0,T;X^{s'+1})),
\end{eqnarray*}
are progressively measurable, for all  $ T>0$, \\
(iii) $\Gamma=\mathbb{P}\circ(u_{0}, \theta_{0})^{-1}$,\\
(iv) for all $\phi\in \mathcal{C}^{\infty}(\mathbb{T}^{3})$ and $t\in [0,T]$, the following holds $\mathbb{P}$ a.s.
\begin{eqnarray*}
\left\{\begin{array}{ll}
u(t)+\int_{0}^{t} Au+\varphi_{R}(\|\nabla u\|_{L^{\infty}})P(u\cdot\nabla )udr=u(0)+\int_{0}^{t}P\theta e_{d}dr+\int_{0}^{t}Pf(u,\theta)d\mathcal{W},\\
\theta(t)+\int_{0}^{t}\varphi_{R}(\|\nabla u\|_{L^{\infty}})(u\cdot\nabla)\theta dr=\theta(0).\\
\end{array}\right.
\end{eqnarray*}
\end{definition}

As mentioned in the introduction, the cut-off function brings trouble in proving the uniqueness which needed for the process of passing from martingale solution to pathwise solution. Therefore, we first establish the existence of martingale solution to \eqref{3.1} in $H^{s'}$ for some fixed integer $s'=s+1$ where the initial data also lies in $H^{s'}$. For this larger $s'$, we can overcome this difficulty of uniqueness which makes it possible to get the pathwise solution.

\subsection{
The approximate solution and a priori estimates}

Suppose $\{\phi_{j}\}_{j=1}^{\infty}$ is the complete orthonormal basis of $L_{\rm div}^2(\mathbb{T}^d)$ of eigenfunctions of the Stokes operator $A$, then let $P_{n}$ be the orthogonal projection from $L_{\rm div}^2(\mathbb{T}^d)$ into span$\{\phi_{1},\cdots\phi_{n}\}$, given by
\begin{equation*}
P_{n}:v\rightarrow v_{n}=\sum\limits_{j=1}^{n}(v,\phi_{j})\phi_{j},~~~\mbox{for all} ~v\in L_{\rm div}^2(\mathbb{T}^d).
\end{equation*}
 To construct the approximate solutions, we apply the mixed method. This technique consists of approximating the momentum equation by the Galerkin scheme and solving the temperature equation directly relative to every approximation solution $u_{n}, n\geq 1$. The approximation scheme is as follows:
\begin{eqnarray}\label{3.3}
\left\{\begin{array}{ll}
du_{n}+Au_{n}dt+\varphi_{R}(\|\nabla u_{n}\|_{L^{\infty}})P_{n}P(u_{n}\cdot\nabla) u_{n}dt\\ \qquad\qquad\qquad\qquad\qquad=P_{n}P\theta_{n} e_{d}dt+P_{n}Pf(u_{n},\theta_n)d\mathcal{W},\\
d\theta_{n}+\varphi_{R}(\|\nabla u_{n}\|_{L^{\infty}})(u_{n}\cdot\nabla )\theta_{n}dt=0,\\
u_{n}(0)=P_{n}u_{0},~~~\theta_{n}(0)=\theta_{0}.
\end{array}\right.
\end{eqnarray}
 At this stage, the approximate velocity field $u_{n}$ is smooth in the space variable $x$, and the equation of temperature admits a classical solution $\theta=\theta(u_{n})$ which shares the same smoothness with the initial data $\theta_{0}$. By (\ref{3.2}), we also have,
\begin{equation}\label{3.4}
\|\nabla \theta_{n}(t,\cdot)\|_{L^{\infty}}\leq {\rm exp}(RT)\|\nabla\theta_{0}\|_{L^{\infty}}\leq R~{\rm exp}(2RT),
\end{equation}
 where the bound is uniform in $n$, and $t\in [0,T]$.  From the equation $(\ref{3.3})_2$, we immediately obtain $\|\theta_{n}\|_{L^{p}}\leq \|\theta_{0}\|_{L^{p}}$ for $t\in [0,T]$ and $p\in[1,\infty]$. With this a priori estimate, the existence of approxiamtion solutions to system (\ref{3.3}) is classical, see \cite{Flan} for further details.

\begin{lemma}\label{lem3.1} Let $s>\frac{d}{2}+1$, $s'=s+1$, $r\geq2$, $\alpha\in [0,\frac{1}{2})$ and suppose that $f$ satisfies conditions (\ref{2.7})-(\ref{2.9}), $u_{0}\in L^{r}(\Omega; X^{s'}), \theta_{0}\in L^{r}(\Omega; H^{s'})$ and for any fixed constant $R$, $\|\nabla\theta_{0}\|_{L^{\infty}}\leq R$ \mbox{a.s.} Then the sequence $\{u_{n}\}_{n\geq 1}$ is uniformly bounded in
\begin{equation*}
L^{r}(\Omega;\mathcal{C}([0,T],X^{s'}))\cap L^{r}(\Omega;C^\alpha([0,T];X^{s'-1}))\cap L^{2}(\Omega\times [0,T];X^{s'+1}),
\end{equation*}
and $\{\theta_{n}\}_{n\geq 1}$ is uniformly bounded in
\begin{equation*}
L^{r}(\Omega;\mathcal{C}([0,T],H^{s'}))\cap L^{r}(\Omega;W^{1,r}(0,T;H^{s'-1})),
\end{equation*}
for any $T>0$. We also have
\begin{eqnarray}
&&\sup_{n\geq 1} \mathbb{E}\left\|u_{n}-\int_{0}^{t}P_{n}P f(u_{n},\theta_{n})d\mathcal{W}\right\|_{W^{1,2}(0,T;X^{s'-1})}^{2}< \infty, \label{3.5}\\
&&\sup_{n\geq 1} \mathbb{E}\|\theta_{n}\|_{W^{1,r}(0,T;H^{s'-1})}^{r}< \infty, \label{3.6}\\
&&\sup_{n\geq 1} \mathbb{E}\left\|\int_{0}^{t}P_{n}P f(u_{n},\theta_{n})d\mathcal{W}\right\|_{C^\alpha([0,T];X^{s'-1})}^{r}< \infty. \label{3.7}
\end{eqnarray}
\end{lemma}

\begin{proof} In order to establish the desired  compactness property in the set of probability measures induced by the distribution of the approximation solution $(u_{n},\theta_{n})$ and pass the limit, we need the uniform estimates on higher integrable of $\|u_{n}\|_{s'}^{2}$ and $\|\theta_{n}\|_{s'}^{2}$. Therefore, for any $r\geq2$, applying the It\^{o} formula to $d(\|u_{n}\|_{s'}^{2})^{\frac{r}{2}}$ yields,
\begin{eqnarray*}
&&d\|u_{n}\|_{s'}^{r}+r\|u_{n}\|_{s'}^{r-2}\|u_{n}\|_{s'+1}^{2}dt
=-r\varphi_{R}(\|\nabla u_{n}\|_{L^{\infty}})\|u_{n}\|_{s'}^{r-2}(u_{n},(u_{n}\cdot\nabla) u_{n})_{s'}dt\\
&&\qquad\qquad\quad+r\|u_{n}\|_{s'}^{r-2}(u_{n},\theta_{n}e_{d})_{s'}dt
+\frac{r}{2}\|u_{n}\|_{s'}^{r-2}\|P_{n}Pf(u_{n},\theta_{n})\|_{L_{2}(H;X^{s'})}^{2}dt\\
&&\qquad\qquad\quad+\frac{r(r-2)}{2}\|u_{n}\|_{s'}^{r-4}(u_{n},Pf(u_{n},\theta_{n}))_{s'}^{2}dt
+r\|u_{n}\|_{s'}^{r-2}(u_{n},Pf(u_{n},\theta_{n}))_{s'}d\mathcal{W}\\
&&\qquad\qquad\quad=(I_{1}+I_{2}+I_{3}+I_{4})dt+I_{5}d\mathcal{W},
\end{eqnarray*}
and
\begin{align*}
d\|\theta_{n}\|_{s'}^{r}&=d(\|\theta_{n}\|_{s'}^{2})^{\frac{r}{2}}=\frac{r}{2}\|\theta_{n}\|_{s'}^{r-2}d(\|\theta_{n}\|_{s'}^{2})\\
&=-r\varphi_{R}(\|\nabla u_{n}\|_{L^{\infty}})\|\theta_{n}\|_{s'}^{r-2}(\theta_{n},u_{n}\cdot \nabla \theta_{n})_{s'}dt
=J_{1}dt.
\end{align*}
Define by $\tau_{K}$ the stopping time
\begin{equation*}
\tau_{K}:={\rm inf}\left\{t\geq 0:\sup_{\gamma\in [0,t]}\|u_{n},\theta_{n}\|_{s'}\geq K\right\}, ~{\rm{for ~any}}~K\in \mathbb{R}^+,
\end{equation*}
if the set is empty, taking $\tau_K=T$. Note that $\tau_K$ is an increasing sequence with $\lim_{K\rightarrow \infty}\tau_K=T$. Hence, taking the integral over time $\gamma\in[0,t\wedge \tau_{K}]$ and then expectation, we have,
\begin{eqnarray*}
&&\quad \mathbb{E}\left(\sup_{\gamma\in [0,t\wedge \tau_{K}]}\|u_{n},\theta_{n}\|_{s'}^{r}\right)+r\mathbb{E}\int_{0}^{t\wedge \tau_{K}}\|u_{n}\|_{s'}^{r-2}\|u_{n}\|_{s'+1}^{2}d\gamma\\
&&\leq \mathbb{E}\int_{0}^{t\wedge \tau_{K}}(|I_{1}|+|I_{2}|+|I_{3}|+|I_{4}|+|J_{1}|)d\gamma
+\mathbb{E}\left(\sup_{\gamma\in[0,t\wedge \tau_{K}]}\left|\int_{0}^{\gamma}I_{5}d\mathcal{W}\right|\right).
\end{eqnarray*}
Next, we estimate $I_{i},i=1,2,3,4,5$ and $J_{1}$ term by term. By Lemma \ref{lem2.2}, we have,
\begin{align}\label{3.8}
 |I_{1}|+|J_{1}|&\leq  C\varphi_{R}(\|\nabla u_{n}\|_{L^{\infty}})\|\nabla u_{n}\|_{L^{\infty}}\|u_{n}\|_{s'}^{r}\nonumber\\
&\quad+C\varphi_{R}(\|\nabla u_{n}\|_{L^{\infty}})((\|\nabla u_{n}\|_{L^{\infty}}
+\|\nabla \theta_{n}\|_{L^{\infty}})\|\theta_{n}\|_{s'}^{r}+\|\nabla \theta_{n}\|_{L^{\infty}}\|u_{n}\|_{s'}^{r}),
\end{align}
where the constant $C=C(s',r, \mathbb{T}^{d})$ is independent of $n$ and $R$. For $I_{2}, I_{3}$ and $I_{4}$, using the H\"{o}lder inequality and condition (\ref{2.9}) yields,
\begin{equation}\label{3.9}
|I_{2}|+|I_{3}|+|I_{4}|\leq C(1+\|u_{n}\|_{s'}^{r}+\|\theta_{n}\|_{H^{s'}}^{r}).
\end{equation}
Regarding the stochastic term, using the Burkholder-Davis-Gundy inequality (\ref{2.4}) and condition (\ref{2.9}),  to obtain,
\begin{eqnarray}\label{3.10}
&&\quad\ \mathbb{E}\left(\sup_{\gamma\in[0,t\wedge \tau_{K}]}\left|\int_{0}^{\gamma}I_{5}d\mathcal{W}\right|\right)
\nonumber\\&&\leq C\mathbb{E}\left(\int_{0}^{t\wedge\tau_{K}}\|u_{n}\|_{s'}^{2r-4}(u_{n},Pf(u_{n},\theta_{n}))_{s'}^{2}d\gamma\right)^{\frac{1}{2}}\nonumber\\
&&\leq C\mathbb{E}\left(\int_{0}^{t\wedge\tau_{K}}\|u_{n}\|_{s'}^{r}(1+\|u_{n},\theta_{n}\|_{s'}^{r})d\gamma\right)^{\frac{1}{2}}\nonumber \\
&&\leq \frac{1}{2}\mathbb{E}\left(\sup_{\gamma\in [0,t\wedge \tau_{K}]}\|u_{n}\|_{s'}^{r}\right)+C\mathbb{E}\int_{0}^{t\wedge \tau_{K}}1+\|u_{n},\theta_{n}\|_{s'}^{r}d\gamma.
\end{eqnarray}
Combining estimates $(\ref{3.4})$ and $(\ref{3.8})$-$(\ref{3.10})$ we have that,
\begin{eqnarray*}
&&\quad \mathbb{E}\left(\sup_{\gamma\in [0,t\wedge \tau_{K}]}\|u_{n}, \theta_{n}\|_{s'}^{r}\right)+\mathbb{E}\int_{0}^{t\wedge \tau_{K}}\|u_{n}\|_{s'}^{r-2}\|u_{n}\|_{s'+1}^{2}d\gamma\\
&&\leq \mathbb{E}\|u_{0},\theta_{0}\|_{s'}^{r}
+C\int_{0}^{t}1+\mathbb{E}\left(\sup_{\xi\in [0,\gamma\wedge \tau_{K}]}\|u_{n},\theta_{n}\|_{s'}^{r}\right)d\gamma,
\end{eqnarray*}
where $C$ is a constant independent of $n$ and $K$ but depends on $(\mathbb{T}^{d}, s',r,R)$. Applying the Gronwall inequality we arrive at,
\begin{equation*}
\mathbb{E}\left(\sup_{\gamma\in [0,t\wedge \tau_{K}]}\|u_{n},\theta_{n}\|_{s'}^{r}\right)+\mathbb{E}\int_{0}^{t\wedge \tau_{K}}\|u_{n}\|_{s'}^{r-2}\|u_{n}\|_{s'+1}^{2}d\gamma\leq C,
\end{equation*}
for any $T\geq 0$ and some positive finite constant $C=C(\mathbb{T}^{d},T,s',r,R,\mathbb{E}\|u_{0},\theta_{0}\|_{s'}^{r})$ which is independent of $n$ and $K$. Since the stopping time is increasing, by the monotone convergence theorem, we have,
\begin{equation}\label{3.11}
\sup_{n}\mathbb{E}\left(\sup_{t\in [0,T]}\|u_{n},\theta_{n}\|_{s'}^{r}\right)+\sup_{n}\mathbb{E}\int_{0}^{T}\|u_{n}\|_{s'}^{r-2}\|u_{n}\|_{s'+1}^{2}dt\leq C.
\end{equation}
Taking $r=2$ in (\ref{3.11}), we obtain that,
\begin{eqnarray*}
 &&\quad \mathbb{E}\left\|u_{n}-\int_{0}^{t}P_{n}P f(u_{n},\theta_{n})d\mathcal{W}\right\|_{W^{1,2}(0,T;H^{s'-1})}^{2}\\
 &&\leq \mathbb{E}\|u_{0}\|_{s'}^{2}+\mathbb{E}\int_{0}^{T}\|u_{n}\|_{s'+1}^{2}dt\\
 &&\quad+C\mathbb{E}\int_{0}^{T}\varphi_{R}(\|\nabla u_{n}\|_{L^{\infty}})\|(u_{n}\cdot\nabla) u_{n}\|_{s'-1}^{2}dt +C\mathbb{E}\int_{0}^{T}\|P_{n}P\theta e_{d}\|_{s'-1}^{2}dt\\
 &&\leq C\mathbb{E}\left(\sup_{t\in [0,T]}\|u_{n}, \theta_{n}\|_{s'}^{2}\right)\leq C.
\end{eqnarray*}
Lemma \ref{lem2.2} along with (\ref{3.4}) implies,
\begin{eqnarray*}
 &&~\mathbb{E}\|\theta_{n}\|_{W^{1,r}(0,T;H^{s'-1})}^{r}\leq \mathbb{E}\|\theta_{0}\|_{H^{s'}}^{r}+
 \mathbb{E}\int_{0}^{T}\varphi_{R}(\|\nabla u_{n}\|_{L^{\infty}})\|(u_{n}\cdot\nabla)\theta_{n}\|_{s'-1}^{2}dt\\
 &&~~\qquad\qquad\qquad\quad\qquad\leq C\mathbb{E}\left(\sup_{t\in [0,T]}\|u_{n}, \theta_{n}\|_{s'}^{r}\right)\leq C,
\end{eqnarray*}
where constant $C$ is independent of $n$ but depends on $(s',T, r, R, \mathbb{E}\|u_{0}, \theta_{0}\|_{s'}^{r})$.
In order to obtain (\ref{3.7}), using (\ref{3.11}) and the Burkholder-Davis-Gundy inequality (\ref{2.4}) again, which yields,
\begin{eqnarray*}
 &&\quad \mathbb{E}\left\|\int_{0}^{t}P_{n}P f(u_{n},\theta_{n})d\mathcal{W}\right\|_{C^{\alpha}([0,T];H^{s'-1})}^{r} \\
 &&\leq\mathbb{E}\left(\sup_{t,s\in[0,T]}\frac{\left\|\int_{s}^{t} P_{n}P f(u_{n},\theta_{n})d\mathcal{W}\right\|^r_{H^{s'-1}}}{|t-s|^{\alpha r}}\right)\\
 &&\leq\mathbb{E}\frac{\left\|\int_{s}^{t} P_{n}P f(u_{n},\theta_{n})d\mathcal{W}\right\|^r_{H^{s'-1}}}{|t-s|^{\alpha r}}+\delta'\\
&&\leq \frac{\mathbb{E}\left(\int_{s}^{t}\|P_{n}P f(u_{n},\theta_{n})\|_{H^{s'-1}}^2d\xi\right)^{\frac{r}{2}}}{|t-s|^{\alpha r}}+\delta'\\
&&\leq \frac{(t-s)^\frac{r}{2}\cdot\mathbb{E}\left(1+\sup_{t\in [0,T]}\|u_{n},\theta_{n}\|_{s'}^{r}\right)}{|t-s|^{\alpha r}}+\delta'\\
&&\leq C|t-s|^{\left(\frac{1}{2}-\alpha\right)r}+\delta'\leq C,
\end{eqnarray*}
where the constant $C$ is independent of $n$ but depends on $(s',T, r, R, \mathbb{E}\|u_{0}, \theta_{0}\|_{s'}^{r})$.

Using the embedding
\begin{eqnarray*}
W^{\beta, r}(0,T;X^{s'-1})\hookrightarrow C^\alpha(0,T;X^{s'-1}), ~{\rm if}~\alpha<\beta-\frac{1}{r},
\end{eqnarray*}
and (\ref{3.5}), (\ref{3.7}), we have
\begin{eqnarray*}
u_n\in L^{r}(\Omega;C^\alpha([0,T];X^{s'-1})),
\end{eqnarray*}
for $\alpha\in [0,\frac{1}{2})$. This completes the proof of the Lemma.
\end{proof}

\subsection{
Tightness and existence of martingale solution}

Let $\{u_{n},\theta_{n}\}_{n\geq 1}$ be the sequence of approximation solutions to system (\ref{3.3}) relative to a fixed stochastic basis $(\Omega, \mathcal{F},\{\mathcal{F}_{t}\}_{t\geq0}, \mathbb{P},\mathcal{W})$ and $\mathcal{F}_{0}$-measurable random variable $(u_{0},\theta_{0})$. We define the path space
\begin{equation*}
\mathcal{X}=\mathcal{X}_{u}\times \mathcal{X}_{\theta}\times \mathcal{X}_{\mathcal{W}},
\end{equation*}
where $\mathcal{X}_{u}=\mathcal{C}([0,T]; X^{s'-1})\cap L^{2}(0,T;X^{s'})$,~~$\mathcal{X}_{\theta}=\mathcal{C}([0,T]; H^{s'-1})$,~~$\mathcal{X}_{\mathcal{W}}=\mathcal{C}([0,T];H)$. Define the probability measures,
\begin{equation}\label{measure}
\mu^{n}=\mu^{n}_{u}\otimes \mu^{n}_{\theta}\otimes \mu_{\mathcal{W}},
\end{equation}
where $\mu^{n}_{u}(\cdot)=\mathbb{P}\{u_{n}\in \cdot\}$, ~$\mu^{n}_{\theta}=\mathbb{P}\{\theta_{n}\in \cdot\}$,~$\mu_{\mathcal{W}}=\mathbb{P}\{\mathcal{W}\in \cdot\}$. Here, the embedding $H^{s'-1}\hookrightarrow H^{1,\infty}$ holds, which is required by passing the limit in cut-off operator. In the following lemma, we show that the set $\{\mu^{n}\}_{n\geq 1}$ is in fact weakly compact.

\begin{lemma}\label{lem3.2}
 The set of the sequence of measures $\{\mu^{n}\}_{n\geq 1}$ defined by (\ref{measure}) is tight on path space $\mathcal{X}$.
\end{lemma}

\begin{proof}
 By applying Lemma \ref{lem2.3}, we deduce that, the embedding
\begin{equation*}
L^{2}(0,T;X^{s'+1})\cap W^{\frac{1}{4},2}(0,T;X^{s'-1})\hookrightarrow L^{2}(0,T;X^{s'}).
\end{equation*}
is compact. For any fixed $K>0$, define the set
\begin{eqnarray*}
 &&B_{K}^{1}:=\bigg\{u\in L^{2}(0,T;X^{s'+1})\cap W^{\frac{1}{4},2}(0,T;X^{s'-1}):\\
 &&\qquad\qquad\qquad \|u\|_{L^{2}(0,T;X^{s'+1})}^{2}+\|u\|_{W^{\frac{1}{4},2}(0,T;X^{s'-1})}^{2}\leq K\bigg\},
\end{eqnarray*}
which is thus compact in $L^{2}(0,T;X^{s'})$. Applying the Chebyshev inequality and the estimates (\ref{3.5}), (\ref{3.7}) and (\ref{3.11}) yield,
\begin{eqnarray}\label{3.12}
&& \mu_{u}^{n}((B_{K}^{1})^{c})=\mathbb{P}\left(\|u_{n}\|_{L^{2}(0,T;X^{s'+1})}^{2}
+\|u_{n}\|_{W^{\frac{1}{4},2}(0,T;X^{s'-1})}^{2}> K\right)\nonumber\\
&&\qquad\qquad\quad\leq \mathbb{P}\left(\|u_{n}\|_{L^{2}(0,T;X^{s'+1})}^{2}> \frac{K}{2}\right)+\mathbb{P}\left(\|u_{n}\|_{W^{\frac{1}{4},2}(0,T;X^{s'-1})}^{2}> \frac{K}{2}\right)\nonumber\\
&&\qquad\qquad\quad\leq \frac{2}{K}\left(\mathbb{E}\int_{0}^{T}\|u_{n}\|_{s'+1}^{2}dt+\mathbb{E}\|u_{n}\|_{W^{\frac{1}{4},2}(0,T;X^{s'-1})}^{2}\right)\leq \frac{C}{K},
\end{eqnarray}
where the constant $C$ is independent of $n$.

Fix any $\alpha\in (0,\frac{1}{2})$, Lemma \ref{lem2.3} gives,
\begin{eqnarray*}
\mathcal{C}([0,T];X^{s'})\cap \mathcal{C}^\alpha([0,T];X^{s'-1})\hookrightarrow\hookrightarrow \mathcal{C}([0,T];X^{s'-1}).
\end{eqnarray*}
Therefore, for any fixed $K\geq 0$, the set,
\begin{eqnarray*}
&&B_{K}^{2}:=\big\{u\in C([0,T];X^{s'})\cap C^{\alpha}([0,T];X^{s'-1}):\\
&&\qquad\qquad\qquad\qquad\|u\|_{C([0,T];X^{s'})}+\|u\|_{C^{\alpha}([0,T];X^{s'-1})}\leq K\big\}
\end{eqnarray*}
is compact in $\mathcal{C}([0,T];X^{s'-1})$. Note that,
\begin{eqnarray*}
\left\{\left\|u_{n}-\int_{0}^{t}P_{n}P f(u_{n},\theta_{n})d\mathcal{W}\right\|_{W^{1,2}(0,T;X^{s'-1})}\right\}\bigcap \left\{\left\|\int_{0}^{t}P_{n}P f(u_{n},\theta_{n})d\mathcal{W}\right\|_{C^{\alpha}([0,T];X^{s'-1})}\right\},
\end{eqnarray*}
is a subset of $ \{u_{n}\in C^{\alpha}([0,T];X^{s'-1})\}$. By the uniform estimates \eqref{3.5}, (\ref{3.7}) and the Chebyshev inequality again, we have,
\begin{eqnarray}\label{3.13}
&&\mu_{u}^{n}((B_{K}^{2})^{c})\leq \mathbb{P}\left({\left\|u_{n}-\int_{0}^{t}P_{n}P f(u_{n},\theta_{n})d\mathcal{W}\right\|_{W^{1,2}(0,T;X^{s'-1})}}> \frac{K}{3}\right)\nonumber\\
&&\qquad\qquad\qquad+\mathbb{P}\left(\left\|\int_{0}^{t}P_{n}P f(u_{n},\theta_{n})d\mathcal{W}\right\|_{C^{\alpha}([0,T];X^{s'-1})}> \frac{K}{3}\right)\nonumber\\
&&\qquad\qquad\qquad+\mathbb{P}\left(\left\|u_n\right\|_{C([0,T];X^{s'})}> \frac{K}{3}\right)\nonumber\\
&&\qquad\qquad\quad\leq \frac{C}{K}\mathbb{E}\left({\left\|u_{n}-\int_{0}^{t}P_{n}P f(u_{n},\theta_{n})d\mathcal{W}\right\|_{W^{1,2}(0,T;X^{s'-1})}}\right.\nonumber\\
&&\qquad\qquad\qquad\left.+\left\|\int_{0}^{t}P_{n}P f(u_{n},\theta_{n})d\mathcal{W}\right\|_{C^{\alpha}([0,T];X^{s'-1})}+\|u_n\|_{C([0,T];X^{s'})}\right)\nonumber\\
&&\qquad\qquad\quad\leq \frac{C}{K},
\end{eqnarray}
where the constant $C$ is independent of $n$. We have that $B_{K}^{1}\cap B_{K}^{2}$ is compact in $\mathcal{C}([0,T];X^{s'-1})\cap L^{2}(0,T;X^{s'})$ for any fixed $K>0$. Using \eqref{3.12} and \eqref{3.13}, we obtain,
\begin{equation*}
\mu_{u}^{n}\left((B_{K}^{1}\cap B_{K}^{2})^{c}\right)\leq \mu_{u}^{n}\left((B_{K}^{1})^{c}\right)+\mu_{u}^{n}\left((B_{K}^{2})^{c}\right)\leq \frac{C}{K}.
\end{equation*}
By a similar argument,  the sequence $\{\mu_{\theta}^{n}\}_{n\geq1}$ is tight in $\mathcal{C}([0,T];H^{s'-1})$.
Finally, we obtain that the sequence $\{\mu^{n}\}_{n\geq 1}$ is tight in $\mathcal{X}$.
\end{proof}

Then, from the tightness property and the classical Skorokhod representation theorem\cite[Theorem 1]{Sko}, we have the following proposition,
\begin{proposition}\label{pro3.1} There exist a subsequence $\{\mu^{n_{k}}\}_{k\geq 1}$, a probability space $(\tilde{\Omega},\tilde{\mathcal{F}},\tilde{\mathbb{P}})$ with $\mathcal{X}$-valued measurable random variables $(\tilde{u}_{n_{k}},\tilde{\theta}_{n_{k}},\tilde{\mathcal{W}}_{n_{k}})$ and $(\tilde{u},\tilde{\theta},\tilde{\mathcal{W}})$ such that\\
(i) $(\tilde{u}_{n_{k}},\tilde{\theta}_{n_{k}},\tilde{\mathcal{W}}_{n_{k}})\rightarrow(\tilde{u},\tilde{\theta},\tilde{\mathcal{W}})$, $\tilde{\mathbb{P}}$ \mbox{a.s.} in the topology of $\mathcal{X}$,\\
(ii) the laws of $(\tilde{u}_{n_{k}},\tilde{\theta}_{n_{k}},\tilde{\mathcal{W}}_{n_{k}})$ and $(\tilde{u},\tilde{\theta},\tilde{\mathcal{W}})$ are given by $\{\mu^{n_{k}}\}_{k\geq 1}$ and $\mu$, respectively,\\
(iii)  $(\tilde{\mathcal{\mathcal{W}}}_{n_{k}})$ is a Wiener process, relative to the filtration $\tilde{\mathcal{F}}_{t}^{n_{k}}$, given by the completion of $\sigma(\tilde{u}_{n_{k}},\tilde{\theta}_{n_{k}},\tilde{\mathcal{W}}_{n_{k}})$,\\
(iv) each pair $(\tilde{u}_{n_{k}},\tilde{\theta}_{n_{k}},\tilde{\mathcal{W}}_{n_{k}})$ satisfies
\begin{eqnarray}
\left\{\begin{array}{ll}
d\tilde{u}_{n_{k}}+A\tilde{u}_{n_{k}}dt+\varphi_{R}(\|\nabla \tilde{u}_{n_{k}}\|_{L^{\infty}})P_{n_{k}}P(\tilde{u}_{n_{k}}\cdot\nabla)\tilde{u}_{n_{k}}dt\\ \qquad\qquad\qquad\qquad=P_{n_{k}}P\tilde{\theta}_{n_{k}} e_{d}dt+P_{n_{k}}Pf(\tilde{u}_{n_{k}},\tilde{\theta}_{n_{k}})d\tilde{\mathcal{W}}_{n_{k}},\\
d\tilde{\theta}_{n_{k}}+\varphi_{R}(\|\nabla \tilde{u}_{n_{k}}\|_{L^{\infty}})(\tilde{u}_{n_{k}}\cdot\nabla)\tilde{\theta}_{n_{k}}dt=0.
\end{array}\right.
\end{eqnarray}
\end{proposition}
\begin{proof} The  three parts (i)-(iii) follow  immediately the Skorokhod representation theorem, and (iv) may be obtained using the same argument as in papers \cite{Breit,DWang}.
\end{proof}

We next show the existence of martingale solution. Before that, We improve the regularity in the space variable of the solution. By (\ref{3.11}) for $\tilde{u}_{n_{k}}$ in the case $r=2$, there exist  $\tilde{u}_{1}\in L^{2}(\tilde{\Omega};L^\infty(0,T;X^{s'}))$ and $\tilde{u}_{2}\in L^{2}(\tilde{\Omega};L^{2}(0,T;X^{s'+1}))$ such that
\begin{equation}\label{3.15}
\tilde{u}_{n_{k}}\rightharpoonup^{\ast} \tilde{u}_{1}~~~ {\rm in} ~L^{2}(\tilde{\Omega};L^\infty(0,T;X^{s'})),
\end{equation}
and
\begin{equation}\label{3.16}
\tilde{u}_{n_{k}}\rightharpoonup\tilde{u}_{2} ~~~{\rm in}~L^{2}(\tilde{\Omega};L^{2}(0,T;X^{s'+1})).
\end{equation}
On the other hand, combining the inequality
\begin{equation*}
\sup_{k}\tilde{\mathbb{E}}\left(\sup_{t\in [0,T]}\|\tilde{u}_{n_{k}}\|_{s'-1}^{r}\right)\leq c\sup_{k}\tilde{\mathbb{E}}\left(\sup_{t\in [0,T]}\|\tilde{u}_{n_{k}}\|_{s'}^{r}\right)<\infty,
\end{equation*}
with Proposition \ref{pro3.1} (i), we have the following result using the Vitali convergence theorem,
\begin{equation}\label{3.17}
\tilde{u}_{n_{k}}\rightarrow \tilde{u}~~~ {\rm in} ~L^{2}(\tilde{\Omega};L^\infty(0,T;X^{s'-1})).
\end{equation}
 For any set $A \subset [0,T]\times \Omega$ measurable and $\phi \in H^{s'}$, we have, from  (\ref{3.15})-(\ref{3.17}),
\begin{equation*}
\tilde{\mathbb{E}}\int_{0}^{T}1_{A}\langle \tilde{u}, \phi\rangle dt=\tilde{\mathbb{E}}\int_{0}^{T}1_{A}\langle \tilde{u}_{1}, \phi\rangle dt=\tilde{\mathbb{E}}\int_{0}^{T}1_{A}\langle \tilde{u}_{2}, \phi\rangle dt,
\end{equation*}
which implies that $\tilde{u}=\tilde{u}_{1}=\tilde{u}_{2}$, a.e. Therefore, we obtain,
\begin{equation*}
\tilde{u}\in L^{2}(\tilde{\Omega};L^\infty(0,T;X^{s'}))\cap L^{2}(\tilde{\Omega};L^{2}(0,T;X^{s'+1})).
\end{equation*}
By a similar argument, we may infer that $\tilde{\theta}\in L^{2}(\tilde{\Omega};L^\infty(0,T;H^{s'}))$.

With these properties established, we can pass the limit by the argument as in \cite{Debussche,Breit}, where the analysis was implemented for the compressible Navier-Stokes equations, primitive equations, respectively. Since the identification of the limit for the case of Boussinesq equations can be proved in the same manner, we omit it. From the estimates and the equation itself, we are able to deduce that the solution is continuous with respect to time $t$ using the \cite[Theorem 3.1]{KR}, see also \cite{Breit} for compressible Navier-Stokes equations.

Up to now, we have established the following proposition,
\begin{proposition}\label{pro3.2} Fix any integer $s>\frac{d}{2}+2$. Suppose that $f$ satisfies conditions (\ref{2.7})-(\ref{2.9}). Then, there exists a martingale solution $\{(\Omega, \mathcal{F}, \{\mathcal{F}_{t}\}_{t\geq 0}, \mathbb{P}),u,\theta,\mathcal{W}\}$ to system (\ref{3.1}) in the sense of definition \ref{adef2.1}.
\end{proposition}

\subsection{
Existence of pathwise solution in $H^{s'}$ }

Following the Yamada-Watanabe type argument, we next establish the pathwise uniqueness and then use the Gy\"{o}ngy-Krylov lemma to recover the convergence a.s. of the approximate solutions on the original probability space.
\begin{proposition}\label{pro3.3}{\rm (Uniqueness)} Fix any $s'=s+1$, $s>\frac{d}{2}+1$. Suppose that $f$ satisfies condition (\ref{2.7}), and  $((\mathcal{S},u_{1},\theta_{1})$, $(\mathcal{S},u_{2},\theta_{2}))$ are two martingale solutions of system (\ref{3.1}) with the same stochastic basis $\mathcal{S}:=(\Omega,\mathcal{F},\{\mathcal{F}_{t}\}_{t\geq 0},\mathbb{P}, \mathcal{W})$. Then if $\mathbb{P}\{(u_{1}(0),\theta_{1}(0))=(u_{2}(0),\theta_{2}(0))\}=1$, then pathwise uniqueness of solutions holds in the sense of Definition 2.1.
\end{proposition}

\begin{proof}
 The difference of two solutions, $v=u_{1}-u_{2}$ and $\eta=\theta_{1}-\theta_{2}$, satisfy,
\begin{eqnarray*}
\left\{\begin{array}{ll}
dv+Avdt+\varphi_{R}(\|\nabla u_{1}\|_{L^{\infty}})P(u_{1}\cdot\nabla)u_{1}dt-\varphi_{R}(\|\nabla u_{2}\|_{L^{\infty}})P(u_{2}\cdot\nabla)u_{2}dt\\ \qquad\qquad\qquad\qquad\qquad\qquad\qquad=P\eta e_{d}dt+P(f(u_{1},\theta_{1})-f(u_{2},\theta_{2}))d\mathcal{W},\\
d\eta+\varphi_{R}(\|\nabla u_{1}\|_{L^{\infty}})(u_{1}\cdot\nabla) \theta_{1} dt-\varphi_{R}(\|\nabla u_{2}\|_{L^{\infty}})(u_{2}\cdot\nabla)\theta_{2}dt=0.
\end{array}\right.
\end{eqnarray*}
Applying the operator $\partial^{\alpha}, |\alpha|\leq s$ to both sides of the system for $v$ and $\eta$ and then applying the It\^{o} formula to $\|\partial^{\alpha}v\|_{L^{2}}^{2}$ give,
\begin{eqnarray*}
&&\quad d\|\partial^{\alpha}v\|_{L^{2}}^{2}+2\|\partial^{\alpha+1}v\|_{L^{2}}^{2}dt\\
&&=-2[\varphi_{R}(\|\nabla u_{1}\|_{L^{\infty}})(\partial^{\alpha}v, \partial^{\alpha}P(u_{1}\cdot\nabla) u_{1})-\varphi_{R}(\|\nabla u_{2}\|_{L^{\infty}})(\partial^{\alpha}v, \partial^{\alpha}P(u_{2}\cdot\nabla) u_{2})]dt\\
&&\quad+2(\partial^{\alpha}v, \partial^{\alpha}P\eta e_{d}) dt+\left|\partial^{\alpha}P(f(u_{1},\theta_{1})-f(u_{2},\theta_{2}))\right|^{2}dt\\
&&\quad+2(\partial^{\alpha}v, \partial^{\alpha}P(f(u_{1},\theta_{1})-f(u_{2},\theta_{2}))) d\mathcal{W}\\
&&=(J_{1}+J_{2}+J_{3})dt+J_{4}d\mathcal{W},
\end{eqnarray*}
and
\begin{equation*}
d\|\partial^{\alpha}\eta\|_{L^{2}}^{2}=-2(\partial^{\alpha}\eta, \varphi_{R}(\|\nabla u_{1}\|_{L^{\infty}})\partial^{\alpha}(u_{1}\cdot\nabla) \theta_{1}-\varphi_{R}(\|\nabla u_{2}\|_{L^{\infty}})\partial^{\alpha}(u_{2}\cdot\nabla) \theta_{2}) dt=I_{1}.
\end{equation*}
Using the mean value theorem for $\varphi_{R}$, the embedding $H^{s}\subset W^{1,\infty}$ and Lemma \ref{lem2.2} yield
\begin{align}\label{3.18}
|J_{1}|&\leq C\sum_{|\alpha|\leq s} |\varphi_{R}(\|\nabla u_{1}\|_{L^{\infty}})-\varphi_{R}(\| \nabla u_{2}\|_{L^{\infty}})|\cdot |(\partial^{\alpha}v, \partial^{\alpha}P(u_{1}\cdot\nabla) u_{1})|\nonumber\\
&\quad+C \sum_{|\alpha|\leq s} (|( \partial^{\alpha}v,\partial^{\alpha}P(v\cdot \nabla) u_{1})|+|( \partial^{\alpha}v,\partial^{\alpha}P(u_{2}\cdot \nabla) v)|)\nonumber\\
&\leq C\left|\|\nabla u_{1}\|_{L^{\infty}}-\|\nabla u_{2}\|_{L^{\infty}}\right|\cdot\|v\|_{s+1}\|u_{1}\|_{s}\|u_{1}\|_{s-1}\nonumber\\
&\quad+\|v\|_{s+1}\|v\|_{s}\|u_1\|_s+\|v\|_{s+1}\|v\|_{s}\|u_2\|_s\nonumber\\ &\leq \|v\|_{s+1}^2+C\|v\|_{s}^{2}(1+\|u_{1}\|_{s}^4+\|u_{2}\|_{s}^2),
\end{align}
and
\begin{align}\label{3.19}
|I_{1}|&\leq C\sum_{|\alpha|\leq s}  |\varphi_{R}(\|\nabla u_{1}\|_{L^{\infty}})-\varphi_{R}(\|\nabla u_{2}\|_{L^{\infty}})|\cdot|(\partial^{\alpha}\eta,\partial^{\alpha}(u_{1}\cdot\nabla)\theta_{1})|\nonumber\\
&\quad+C\sum_{|\alpha|\leq s} (|(\partial^{\alpha}(v\cdot\nabla) \theta_{1},\partial^{\alpha}\eta )|+|(\partial^{\alpha}(u_{2}\cdot\nabla)\eta,\partial^{\alpha}\eta )|)\nonumber\\
&\leq \|v\|_s\|\eta\|_{s}\|u_{1}\|_{s}\|\theta_{1}\|_{s+1}+\|\eta\|_{s}(\|v\|_{L^{\infty}}\|\theta_{1}\|_{s+1}+\|\nabla \theta_{1}\|_{L^{\infty}}\|v\|_{L^{\infty}})\nonumber\\
&\quad+ \|\eta\|_{s}(\|\nabla u_{2}\|_{L^{\infty}}\|\eta\|_{s}+\|\nabla \eta\|_{L^{\infty}}\|u_{2}\|_{s})\nonumber\\
&\leq \|\eta\|_{s}^{2}(\|u_{1}\|_{s}\|\theta_{1}\|_{s+1}+\|u_{2}\|_{s})+\|\eta, v\|_{s}^{2}(\|\theta_{1}\|_{s+1}+\|\theta_{1}\|_{s}).
\end{align}
For $J_{2}$ and $J_{3}$, applications of the H\"{o}lder inequality and condition $(\ref{2.7})$ give,
\begin{equation*}\label{3.20}
|J_{2}|+|J_{3}|\leq C\|\eta, v\|_{s}^{2}.
\end{equation*}
For the term $J_{4}$ we apply the Burkholder-Davis-Gundy inequality similar to $(\ref{3.10})$ to obtain,
\begin{eqnarray}\label{3.21}
&&\mathbb{E}\left(\sup_{\gamma\in [0,t]} \left|\int_{0}^{\gamma\wedge \tau_{K}}J_{4}d\mathcal{W}\right|\right)\leq C\mathbb{E}\left(\int_{0}^{t\wedge \tau_{K}}\left|J_4\right|^{2}d\gamma\right)^{\frac{1}{2}}\nonumber\\
&&~\quad\qquad\qquad\qquad\qquad\qquad\leq C\mathbb{E}\left(\int_{0}^{t\wedge \tau_{K}}\|\partial^{\alpha}v\|_{L^{2}}^{2}\|f(u_{1},\theta_{1})-f(u_{2},\theta_{2})\|_{L_{2}
(H;H^{s})}^{2}d\gamma\right)^{\frac{1}{2}}\nonumber\\
&&~\quad\qquad\qquad\qquad\qquad\qquad\leq \frac{1}{2}\mathbb{E}\left(\sup_{\gamma\in[0,t\wedge \tau_{K}]}\|\partial^{\alpha}v\|_{L^{2}}^{2}\right)+C\mathbb{E}\int_{0}^{t\wedge \tau_{K}}\|v, \eta\|_{s}^{2}d\gamma,
\end{eqnarray}
where the collection of stopping times can be defined as
\begin{equation*}
\tau_{K}:=\inf\left\{t\geq 0: \sup_{\gamma\in [0,t]}(\|u_{1},u_{2}\|_{s}^{2}+\|\theta_{1}\|_{s+1}^{2})\geq K\right\}.
\end{equation*}
We have $\tau_{K}\rightarrow \infty$ a.s. as $K\rightarrow \infty$ due to a priori estimate $(\ref{3.11})$ and the assumption on $s'$.
Combining estimates $(\ref{3.18})$-$(\ref{3.21})$ and summing over all $\alpha$ with $|\alpha| \leq s$, we have,
\begin{eqnarray*}
&&\quad \mathbb{E}\left(\sup_{\gamma\in [0,t\wedge \tau_{K}]}\|v, \eta\|_{s}^{2}\right)+\mathbb{E}\int_{0}^{t\wedge \tau_{K}}\|v\|_{X^{s+1}}^{2}d\gamma\\
&&\leq \mathbb{E}\int_{0}^{t\wedge \tau_{K}}\|\eta, v\|_{s}^{2}(1+\|\theta_{1}\|_{s+1}^2+\|u_{1}\|_{s}^4+\|u_{2}\|_{s}^2)d\gamma\\
&&\leq C \int_{0}^{t}\mathbb{E}\left(\sup_{r\in[0,\gamma\wedge \tau_{K}]}\|v, \eta\|_{s}^{2}\right)d\gamma,
\end{eqnarray*}
where constant $C$ depends on $K$ via the definition of the stopping time $\tau_{K}$. By the Gronwall inequality and the monotone convergence theorem, we infer that,
\begin{equation*}
\mathbb{E}\left(\sup_{t\in [0,T]}\|v, \eta\|_{s}^{2}\right)+\mathbb{E}\int_{0}^{T}\|v\|_{s+1}^{2}dt=0,
\end{equation*}
for every $T>0$. The uniqueness follows.
\end{proof}

The following proposition and its proof can be found in \cite{Krylov}.
\begin{proposition}\label{pro3.4}
Let $X$ be a complete separable metric space and suppose that $\{Y_{n}\}_{n\geq0}$ is a sequence of $X$-valued random variables on a probability space $(\Omega,\mathcal{F},\mathbb{P})$. Let $\{\mu_{m,n}\}_{m,n\geq1}$ be the set of joint laws of $\{Y_{n}\}_{n\geq1}$, that is
\begin{equation*}
\mu_{m,n}(E):=\mathbb{P}\{(Y_{n},Y_{m})\in E\},~~~E\in\mathcal{B}(X\times X).
\end{equation*}
Then $\{Y_{n}\}_{n\geq1}$ converges in probability if and only if for every subsequence of the joint probability laws $\{\mu_{m_{k},n_{k}}\}_{k\geq1}$, there exists a further subsequence that converges weakly to a probability measure $\mu$ such that
\begin{equation*}
\mu\{(u,v)\in X\times X: u=v\}=1.
\end{equation*}
\end{proposition}

We denote by $\mu_{n,m}$ the joint law of
\begin{equation*}
(u_{n},\theta_{n};u_{m},\theta_{m})~~~~ {\rm on~ the ~path ~space}~\mathcal{X}=\mathcal{X}_{u}\times \mathcal{X}_{\theta}\times \mathcal{X}_{u}\times \mathcal{X}_{\theta},
\end{equation*}
where $\{u_{n},\theta_{n};u_{m},\theta_{m}\}_{n,m\geq 1}$ is a sequence of approximation solutions to system \eqref{3.3} relative to the given stochastic basis $\mathcal{S}$, and denote by $\mu_{\mathcal{W}}$ the law of $\mathcal{W}$ on $\mathcal{X}_{\mathcal{W}}$. We introduce the extended phase space,
\begin{equation*}
\tilde{\mathcal{X}}=\mathcal{X}_{u}\times \mathcal{X}_{\theta}\times \mathcal{X}_{u}\times \mathcal{X}_{\theta}\times\mathcal{X}_{\mathcal{W}},
\end{equation*}
and denote by $\nu_{n,m}$ the joint law of $(u_{n},\theta_{n};u_{m},\theta_{m},\mathcal{W})~~~~{\rm on}~~\tilde{\mathcal{X}}$. Using a similar argument as in the proof of Lemma \ref{lem3.2}, we obtain the following result.
\begin{lemma}\label{lem3.3} The set $\{\nu_{n,m}\}_{n,m\geq 1}$ is tight on $\tilde{\mathcal{X}}$.
\end{lemma}

For any subsequence $\{\nu_{n_{k},m_{k}}\}_{k\geq 1}$, by the Skorokhod representation theorem, there exists another probability space $(\tilde{\Omega},\tilde{\mathcal{F}},\tilde{\mathbb{P}})$ and $\tilde{\mathcal{X}}$-valued random variables
\begin{eqnarray*}
(\tilde{u}_{n_{k}},\tilde{\theta}_{n_{k}};\tilde{u}_{m_{k}},\tilde{\theta}_{m_{k}};\tilde{\mathcal{W}}_{k}),~ {\rm and}~(\tilde{u}_{1},\tilde{\theta}_{1};\tilde{u}_{2},\tilde{\theta}_{2};\tilde{\mathcal{W}})
\end{eqnarray*}
such that
\begin{eqnarray*}
 \tilde{\mathbb{P}}\{(\tilde{u}_{n_{k}},\tilde{\theta}_{n_{k}};\tilde{u}_{m_{k}},\tilde{\theta}_{m_{k}};\tilde{\mathcal{W}}_{k})\in \cdot\}=\nu_{n_{k},m_{k}}(\cdot),
\end{eqnarray*}
 and
\begin{eqnarray*}
 (\tilde{u}_{n_{k}},\tilde{\theta}_{n_{k}};\tilde{u}_{m_{k}},\tilde{\theta}_{m_{k}};\tilde{\mathcal{W}}_{k})\rightarrow (\tilde{u}_{1},\tilde{\theta}_{1};\tilde{u}_{2},\tilde{\theta}_{2};\tilde{\mathcal{W}}),~~\tilde{\mathbb{P}} ~a.s.
\end{eqnarray*}
in the topology of $\tilde{\mathcal{X}}$. Analogously, this theorem can be applied to both
\begin{equation*}
(\tilde{u}_{n_{k}},\tilde{\theta}_{n_{k}},\tilde{\mathcal{W}}_{k}),
~~(\tilde{u}_{1},\tilde{\theta}_{1},\tilde{\mathcal{W}}), \hspace{.3cm} \text{and} \hspace{.3cm}
(\tilde{u}_{m_{k}},\tilde{\theta}_{m_{k}},\tilde{\mathcal{W}}_{k}),~~(\tilde{u}_{2},\tilde{\theta}_{2},\tilde{\mathcal{W}})
\end{equation*}
to show that $(\tilde{u}_{1},\tilde{\theta}_{1},\tilde{\mathcal{W}})$ and $(\tilde{u}_{2},\tilde{\theta}_{2},\tilde{\mathcal{W}})$ are martingale solutions relative to the same stochastic basis $\mathcal{S}:=(\tilde{\Omega},\tilde{\mathcal{F}},\tilde{\mathbb{P}},\{\tilde{\mathcal{F}}_{t}\}_{t\geq 0},\tilde{\mathcal{W}})$. Defining $\mu(\cdot)=\tilde{\mathbb{P}}\{(\tilde{u}_{1},\tilde{u}_{2};\tilde{\theta}_{1},\tilde{\theta}_{2})\in \cdot\}$, due to the convergence a.s. in $\mathcal{X}$, we have $\mu_{n,m}\rightharpoonup \mu$. Proposition \ref{pro3.3} implies that $\mu\{(u_{1},\theta_{1};u_{2},\theta_{2})\in \mathcal{X}:(u_{1},\theta_{1})=(u_{2},\theta_{2})\}=1$. Also since we have the uniqueness in $ H^{s}=H^{s'-1}$. Therefore, Proposition \ref{pro3.4} can be used to deduce that the sequence $(u_{n},\theta_{n})$ defined on the original probability space $(\Omega,\mathcal{F},\mathbb{P})$ converges a.s. in the topology of $\mathcal{X}_{u}\times \mathcal{X}_{\theta}$ to random variable $(u,\theta)$.
Again by the method from above, we may show that $(u,\theta)$ is a pathwise solution of (3.1).

Next, define the stopping time,
\begin{equation*}
\tau=\inf \{t\geq 0:\|u\|_{1,\infty}\geq R\}.
\end{equation*}
Hence, relative to the fixed stochastic basis $\mathcal{S}$, $(u,\theta,\tau)$ is a local pathwise solution to the system (\ref{Equ1.1}), for which $u(\cdot \wedge \tau)\in L^{2}(\Omega;\mathcal{C}([0,\infty);X^{s'}))$, $u 1_{t\leq \tau}\in L^{2}(\Omega;L^{2}_{loc}(0,\infty;X^{s'+1})$, $\theta(\cdot\wedge \tau)\in L^{2}(\Omega; \mathcal{C}([0,\infty);H^{s'}$)), and (\ref{2.11}) holds for every $t\geq 0$. In order to show that $\tau >0$ and to loosen the integrability in the random element $w$, the initial data has to be truncated. Actually, using the technique given in \cite{Breit,Glatt}, we can release the restriction on initial data to general case.

\subsection{
Extending to the maximal pathwise solution} The proof is standard, we refer the reader to \cite{Breit,Glatt-Holtz}.

\bigskip

\section{The existence of pathwise solution in $H^{s}$}\label{sec4}
\setcounter{equation}{0}

In this section, we extend the range of the regularity index of space to any integer $s>\frac{d}{2}+1$. Inspired by \cite{Lai,Masmoudi}, we shall adopt a density and stability argument, the pathwise solution evolving in $H^{s'}$ obtained in Section 3 will be used for the approximate solutions. To extract a strongly convergent subsequence and overcome the difficulty of compactness, a pairwise comparison technique introduced in \cite{Glatt,Rozovskii} will be employed.

First, we review some basic properties of a class of smoothing operators $\rho_{\epsilon}$ which was constructed in \cite{Bona} on the whole space, in \cite{Tang} on the torus.
\begin{lemma}\label{lem4.1}
Let $s\geq 0$. For every $\epsilon>0$, the operator $\rho_{\epsilon}$ maps $H^{s}(\mathbb{T}^{d})$ into $H^{s'}(\mathbb{T}^{d})$ where $s'=s+1$ and has the following properties,

(i) The collection $\{\rho_{\epsilon}\}_{\epsilon >0}$ is uniformly bounded in $H^{s}(\mathbb{T}^{d})$ independent of $\epsilon$, i.e. there exists a positive constant C=C(s) such that,
\begin{equation*}
\|\rho_{\epsilon}f\|_{s}\leq C\|f\|_{s},~~~f\in H^{s}(\mathbb{T}^{d}).
\end{equation*}

(ii) For every $\epsilon>0$, if $s\geq 1$  then for $f\in H^{s}(\mathbb{T}^{d})$,
\begin{equation*}
\|\rho_{\epsilon}f\|_{s}\leq \frac{C}{\epsilon}\|f\|_{s-1}, \hspace{.3cm} \text{and} \hspace{.3cm}\|\rho_{\epsilon}f-f\|_{s-1}\leq C\epsilon\|f\|_{s}.
\end{equation*}

(iii) Sequence $\rho_{\epsilon}f$ converges to $f$, for $f\in H^{s}(\mathbb{T}^{d})$, that is,
\begin{equation*}
\lim_{\epsilon\rightarrow 0}\|\rho_{\epsilon}f-f\|_{s}=0, \hspace{.3cm} \text{and} \hspace{.3cm}\lim_{\epsilon\rightarrow 0}\frac{1}{\epsilon}\|\rho_{\epsilon}f-f\|_{s-1}=0.
\end{equation*}
 In particular, if $\{f_{k}\}_{k\geq 1}$ is a sequence of functions in $H^{s}(\mathbb{T}^{d})$ that converges in $H^{s}(\mathbb{T}^{d})$, then for $s\geq 1$,
\begin{equation*}
\lim_{\epsilon\rightarrow 0}\sup_{k\geq 1}\|\rho_{\epsilon}f_{k}-f_{k}\|_{s}=0,\hspace{.3cm} \text{and} \hspace{.3cm}
\lim_{\epsilon\rightarrow 0}\frac{1}{\epsilon}\sup_{k\geq 1}\|\rho_{\epsilon}f_{k}-f_{k}\|_{s-1}=0.
\end{equation*}
\end{lemma}
First, we use the operator given in Lemma \ref{lem4.1} to mollify the initial data, obtaining a sequence of smooth initial data,
\begin{equation*}
u_{0}^{j}=\rho_{j^{-1}}u_{0},~~\theta_{0}^{j}=\rho_{j^{-1}}\theta_{0},
\end{equation*}
for $j\geq 1$. Relative to the initial data $\{u_{0}^{j},\theta_{0}^{j}\}_{j\geq 1}$, we can obtain a sequence of maximal, pathwise solutions $(u_{j},\theta_{j})$ as the argument given in Section 3. In order to be able to apply the Lemma given below, we restrict the initial data $\|u_{0},\theta_{0}\|_{s}\leq M$ for any fixed $M$, by Lemma \ref{lem4.1} (i), we know the sequence of initial data $\{u_{0}^{j},\theta_{0}^{j}\}_{j\geq 1}$ is also bounded uniformly in $j$, thus,
\begin{equation*}
\sup_{j\geq 1}\|u_{0}^{j},\theta_{0}^{j}\|_{s}\leq C \|u_{0},\theta_{0}\|_{s}\leq C M,
\end{equation*}
where $C=C(s)$ is constant. Actually, this restriction can be generalized to that of $(u_{0},\theta_{0})\in X^{s}\times H^{s}$ a.s. by a cutting argument.

\begin{lemma}{\rm \cite[lemma 5.1]{Glatt}}\label{lem4.2}
Fix any $T\geq 0$ and define the collection of stopping times,
\begin{equation}\label{4.1}
\tau_{j}^{T}=\inf\left\{t\geq 0:\sup_{r\in[0,t]}\|u_{j},\theta_{j}\|_{s}^{2}+\int_{0}^{t}\|u_{j}\|_{s+1}^{2}dr\geq 1+\|u_{0}^{j},\theta_{0}^{j}\|_{s}^{2}\right\}\wedge T,
\end{equation}
and take $\tau_{j,k}^{T}:=\tau_{j}^{T}\wedge \tau_{k}^{T}$. Suppose that,
\begin{equation}\label{4.2}
\lim_{j\rightarrow\infty}\sup_{k\geq j}\mathbb{E}\left(\sup_{t\in[0,\tau_{j,k}^{T}]}\|u_{k}-u_{j},\theta_{k}-\theta_{j}\|_{s}^{2}
+\int_{0}^{\tau_{j,k}^{T}}\|u_{k}-u_{j}\|_{s+1}^{2}dt\right)=0,
\end{equation}
and
\begin{equation}\label{4.3}
\lim_{S\rightarrow 0}\sup_{j\geq 1}\mathbb{P}\left(\sup_{t\in[0,\tau_{j}^{T}\wedge S]}\|u_{j},\theta_{j}\|_{s}^{2}+2\int_{0}^{\tau_{j}^{T}\wedge S}\|u_{j}\|_{s+1}^{2}dt>\|u_{0}^{j},\theta_{0}^{j}\|_{s}^{2}+1\right)=0,
\end{equation}
then there exists a stopping time $\tau$ with $\mathbb{P}\{0< \tau\leq T\}=1$, the predictable processes $u(\cdot\wedge \tau)\in L^\infty(0,\infty;X^{s})\cap L^{2}_{loc}(0,\infty;X^{s+1})$ and $ \theta(\cdot\wedge \tau)\in L^\infty(0,\infty;H^{s})$ satisfy,
\begin{equation}\label{4.4}
\sup_{t\in[0,\tau]}\|u_{j_{i}}-u,\theta_{j_{i}}-\theta\|_{s}^{2}+\int_{0}^{\tau}\|u_{j_{i}}-u\|_{s+1}^{2}dt\rightarrow 0, ~~~~\mbox{a.s.}
\end{equation}
for some subsequence $j_{i}\rightarrow \infty$. Moreover,
\begin{equation}\label{4.5}
\sup_{t\in[0,\tau]}\|u,\theta\|_{s}^{2}+\int_{0}^{\tau}\|u\|_{s+1}^{2}dt\leq 1+\sup_{j}\|u_{0}^{j},\theta_{0}^{j}\|_{s}^{2}, ~~~~\mbox{a.s.}
\end{equation}
\end{lemma}

In order to obtain the convergence a.s. needed for Theorem 2.1, we show that conditions $(\ref{4.2})$ and $(\ref{4.3})$ hold. The proof follows the idea of \cite{Glatt-Holtz}.

\begin{proof}[Proof of \eqref{4.2} and \eqref{4.3}]
The difference of the solutions, $v=u_{j}-u_{k}$ and $\eta=\theta_{j}-\theta_{k}$, satisfy,
\begin{eqnarray*}
\left\{\begin{array}{ll}
dv+Avdt+P(u_{j}\cdot\nabla) u_{j}dt-P(u_{k}\cdot\nabla) u_{k}dt\\ \qquad\qquad\qquad\qquad\qquad\qquad=P\eta e_{2}dt+P(f(u_{j},\theta_{j})-f(u_{k},\theta_{k}))d\mathcal{W},\\
d\eta+(u_{j}\cdot\nabla) \theta_{j}dt-(u_{k}\cdot\nabla) \theta_{k}dt=0.
\end{array}\right.
\end{eqnarray*}
Applying the It\^{o} formula to $\|\partial^{\alpha}v\|_{L^{2}}^{2}$, we have
\begin{align*}
&\quad d\|\partial^{\alpha}v\|_{L^{2}}^{2}+2\|\partial^{\alpha+1}v\|_{L^{2}}^{2}dt\nonumber\\&=-2[(\partial^{\alpha}v, \partial^{\alpha}P(u_{j}\cdot\nabla) u_{j})-(\partial^{\alpha}v, \partial^{\alpha}P(u_{k}\cdot\nabla) u_{k})]dt\\
&\quad+2(\partial^{\alpha}v, \partial^{\alpha}P\eta e_{d}) dt+ \|\partial^{\alpha}P(f(u_{j},\theta_{j})-f(u_{k},\theta_{k}))\|_{L_2(H,L^2)}^{2}dt\\
&\quad+2(\partial^{\alpha}v, \partial^{\alpha}P(f(u_{j},\theta_{j})-f(u_{k},\theta_{k}))) d\mathcal{W}\\&=(J_{1}+J_{2}+J_{3})dt+J_{4}d\mathcal{W},
\end{align*}
and
\begin{equation*}
d\|\partial^{\alpha}\eta\|_{L^{2}}^{2}=-2(\partial^{\alpha}\eta, \partial^{\alpha}(u_{j}\cdot\nabla \theta_{j})-\partial^{\alpha}(u_{k}\cdot\nabla \theta_{k}))dt=I_{1}dt.
\end{equation*}
For the nonlinear terms $J_{1}$ and $I_{1}$, we use Lemma \ref{lem2.1} to obtain,
\begin{align*}
\sum_{|\alpha|\leq s}|J_{1}+I_1|&\leq C\sum_{|\alpha|\leq s}|(\partial^{\alpha}v, \partial^{\alpha}P(v\cdot \nabla) u_{j}+\partial^{\alpha}P(u_{k}\cdot \nabla) v)\\&\qquad+( \partial^{\alpha}\eta,  \partial^{\alpha}(v\cdot \nabla) \theta_{j}+ \partial^{\alpha}(u_{k}\cdot \nabla)
\eta)|\\
&\leq C\|v\|_{s}^{2}(\|u_{k}\|_{s}^2+\|u_{j}\|_{s}^2+1)+\|v\|_{s+1}^{2}\\
&\quad+\|v\|_{s-1}^{2}\|\theta_{j}\|_{s+1}^{2}+(\|\eta\|_{s}^{2}+\|v\|_{s}^{2})(1+\|\theta_{j}\|_{s}+\|\theta_{k}\|_{s}),
\end{align*}
where the constant $C=C(s, \mathbb{T}^{d})$ is independent of $j$ and $k$. For $J_{2}$ and $J_{3}$, using the H\"{o}lder inequality and condition (\ref{2.7}), we easily get
\begin{equation*}
|J_{2}|+|J_{3}|\leq C \|v,\eta\|_{s}^{2}.
\end{equation*}
For the stochastic term $J_{4}$, similar to the estimate (\ref{3.21}), we have for any stopping time $\tau$,
\begin{eqnarray*}
 &&\mathbb{E}\left(\sup_{r\in [0,\tau]}\left|\int_{0}^{r}J_{4}d\mathcal{W}\right|\right)\leq C\mathbb{E}\left(\int_{0}^{\tau}\sum_{l\geq 1}J_4^{2}dr\right)^{\frac{1}{2}}\\
&&\qquad\qquad\qquad\qquad\qquad\leq \frac{1}{2} \mathbb{E}\left(\sup_{r\in[0,\tau]}\|\partial^{\alpha}v\|_{L^{2}}^{2}\right)+C \mathbb{E}\int_{0}^{\tau}\|v, \eta\|_{s}^{2}dr.
\end{eqnarray*}
Combining the above estimates   and the definition of $\tau_{j,k}^{T}$, to conclude,
\begin{eqnarray*}
&&\quad \mathbb{E}\left(\sup_{r\in[0,\tau_{j,k}^{T}\wedge t]}\|v, \eta\|_{s}^{2}
\right)+\mathbb{E}\int_{0}^{\tau_{j,k}^{T}\wedge t}\|v\|_{s+1}^{2}dr\\&& \leq \mathbb{E}\|v_{0},\eta_{0}\|_{s}^{2}+C\mathbb{E}\int_{0}^{\tau_{j,k}^{T}\wedge t}\|v\|_{s-1}^{2}\|\theta_{j}\|_{s+1}^{2}dr\\
&&\quad+C \mathbb{E}\int_{0}^{\tau_{j,k}^{T}\wedge t}\|\eta, v\|_{s}^{2}(1+\|\theta_{j}\|_{s}+\|\theta_{k}\|_{s}+\|u_{j}\|_{s}^2+\|u_{k}\|_{s}^2)dr\\
&&\leq C\mathbb{E}\|v_{0},\eta_{0}\|_{s}^{2}
+C\int_{0}^{t}\mathbb{E}\left(\sup_{\xi\in [0,\tau_{j,k}^{T}\wedge r]}(\|v,\eta\|_{s}^{2}+\|v\|_{s-1}^{2}\|\theta_{j}\|_{s+1}^{2})\right)dr.
\end{eqnarray*}
The Gronwall lemma yields,
\begin{equation*}
\mathbb{E}\left(\sup_{t\in[0,\tau_{j,k}^{T}]}\|v,\eta\|_{s}^{2}\right)+\mathbb{E}\int_{0}^{\tau_{j,k}^{T}\wedge t}\|v\|_{s+1}^{2}dr\leq C \mathbb{E}\|v_{0},\eta_{0}\|_{s}^{2}+C\mathbb{E}\left(\sup_{t\in [0,\tau_{j,k}^{T}]}\|v\|_{s-1}^{2}\|\theta_{j}\|_{s+1}^{2}\right),
\end{equation*}
where $C=C(s,T,M,\mathbb{T}^{d})$ is a positive constant independent of $j,k$. Therefore, (\ref{4.2}) will follow once we show that
\begin{equation}\label{4.6}
\lim_{j\rightarrow \infty}\sup_{k\geq j}\mathbb{E}\left(\sup_{t\in[0,\tau_{j,k}^{T}]}\|v\|_{s-1}^{2}\|\theta_{j}\|_{s+1}^{2}\right)=0.
\end{equation}
 Due to the coupled construction of the system, consequently, the terms $\|\eta\|_{{s-1}}^{2}(\|u_{j}\|_{s+1}^{2}+\|\theta_{j}\|_{s+1}^{2})$, $\|v\|_{s-1}^{2}\|u_{j}\|_{s+1}^{2}$ will appear when we build the estimates after applying the It\^{o} product formula to function $\|v\|_{s-1}^{2}\|\theta_{j}\|_{s+1}^{2}$. Therefore, we need to show that
\begin{equation}\label{4.7}
\lim_{j\rightarrow \infty}\sup_{k\geq j}\mathbb{E}\left(\sup_{t\in[0,\tau_{j,k}^{T}]}(\|v\|_{{s-1}}^{2}+\|\eta\|_{s-1}^{2})(\|u_{j}\|_{s+1}^{2}+\|\theta_{j}\|_{s+1}^{2})\right)=0.
\end{equation}
For convenience, we write,
\begin{align*}
d\|\partial^\alpha u_j\|_{L^2}^2+2\|\partial^{\alpha+1} u_j\|_{L^2}^2dt&=2\bigg((\partial^{\alpha}u_{j}, -\partial^{\alpha}P(u_{j}\cdot\nabla) u_{j}) +(\partial^{\alpha}u_{j},\partial^{\alpha}\theta_{j}e_{d})\nonumber\\
&\quad+\frac{1}{2}\|\partial^{\alpha}Pf(u_{j},\theta_{j})\|^{2}_{L_{2}(H;L^{2})}\bigg)dt+2(\partial^{\alpha}u_{j},\partial^{\alpha}Pf(u_{j},\theta_{j})) d\mathcal{W}\nonumber\\
&=(\mathcal{J}_1+\mathcal{J}_2+\mathcal{J}_3)dt+\mathcal{J}_4d\mathcal{W},
\end{align*}
and
\begin{eqnarray*}
d\|\partial^\alpha \theta_j\|_{L^2}^2=2(\partial^{\alpha}\theta_{j}, -\partial^{\alpha}(u_{j}\cdot\nabla) \theta_{j})dt=\mathcal{I}_1dt.
\end{eqnarray*}
By the It\^{o} product formula, we have,
\begin{eqnarray}\label{4.8}
&&\quad d\|v\|_{{s-1}}^{2}\|u_{j}\|_{s+1}^{2}=\|v\|_{{s-1}}^{2}d\|u_{j}\|_{s+1}^{2}+\|u_{j}\|_{s+1}^{2}d\|v\|_{{s-1}}^{2}
+d\|v\|_{{s-1}}^{2}d\|u_{j}\|_{s+1}^{2}\nonumber\\
&&=-2(\|v\|_{{s-1}}^{2}\|u_{j}\|_{s+2}^{2}+\|u_{j}\|_{s+1}^{2}\|v\|_{s}^{2})dt\nonumber \\
&&\quad+2\|v\|_{{s-1}}^{2}\sum_{|\alpha|\leq s+1}(\mathcal{J}_1+\mathcal{J}_2+\mathcal{J}_3)dt+2\sum_{|\alpha|\leq s+1}\|v\|_{{s-1}}^{2}\mathcal{J}_4 d\mathcal{W}\nonumber\\
&&\quad+2\|u_{j}\|_{s+1}^{2}\sum_{|\alpha|\leq s-1}(J_1+J_2+J_3)dt+2\sum_{|\alpha|\leq s-1}\|u_{j}\|_{s+1}^{2}J_4 d\mathcal{W}\nonumber\\
&&\quad+4\left(\sum_{|\alpha|\leq s-1}J_4\cdot\sum_{|\alpha|\leq s+1}\mathcal{J}_4\right)dt,
\end{eqnarray}
and
\begin{eqnarray}\label{4.9}
&&\quad d\|\eta\|_{{s-1}}^{2}\|u_{j}\|_{s+1}^{2}=\|\eta\|_{{s-1}}^{2}d\|u_{j}\|_{s+1}^{2}+\|u_{j}\|_{s+1}^{2}d\|\eta\|_{{s-1}}^{2}
+d\|\eta\|_{{s-1}}^{2}d\|u_{j}\|_{s+1}^{2}\nonumber\\
&&=-2\|\eta\|_{{s-1}}^{2}\|u_{j}\|_{s+2}^{2}dt+2\|\eta\|_{{s-1}}^{2}\sum_{|\alpha|\leq s+1}(\mathcal{J}_1+\mathcal{J}_2+\mathcal{J}_3)dt+2\sum_{|\alpha|\leq s+1}\|\eta\|_{{s-1}}^{2}\mathcal{J}_4 d\mathcal{W}\nonumber\\
&&\quad-2\|u_{j}\|_{s+1}^{2}\sum_{|\alpha|\leq s-1}I_1 dt.
\end{eqnarray}
By Lemma \ref{lem2.2}, conditions (\ref{2.7}), and the H\"{o}lder inequality, we have,
\begin{eqnarray*}
&&\left|2\|\eta\|_{{s-1}}^{2}\sum_{|\alpha|\leq s+1}\mathcal{J}_1\right|\leq C\|\eta\|_{{s-1}}^{2}\|u_{j}\|_{s+1}^{2}\|u_{j}\|_{s},\\
&&\left|2\|\eta\|_{{s-1}}^{2}\sum_{|\alpha|\leq s+1}(\mathcal{J}_2+\mathcal{J}_3)\right|\leq C\|u_{j}\|_{s+1}^{2}(\|\eta\|_{{s-1}}^{2}+\|v\|_{s-1}^{2})(\|\theta_{j}\|_{s}+\|u_{k}\|_{s}),\\
&&\left|2\|v\|_{{s-1}}^{2}\sum_{|\alpha|\leq s+1}(\mathcal{J}_1+\mathcal{J}_2+\mathcal{J}_3)\right|\leq C \|v\|_{s-1}^{2}+C\|v\|_{s-1}^{2}(\|u_{j}\|_{s+1}^{2}+\|\theta_{j}\|_{s+1}^{2}),\\
&&\left|2\|u_{j}\|_{s+1}^{2}\sum_{|\alpha|\leq s-1}(J_2+J_3)+4\left(\sum_{|\alpha|\leq s-1}J_4\cdot\sum_{|\alpha|\leq s+1}\mathcal{J}_4\right)\right|\\
&&\leq C (\|v\|_{s-1}^{2}+\|\eta\|_{s-1}^{2})(1+\|u_{j}\|_{s+1}^{2}+\|\theta_{j}\|_{s+1}^{2}),\\
&&\left|2\|u_{j}\|_{s+1}^{2}\sum_{|\alpha|\leq s-1}J_1\right|\leq C \|u_{j}\|_{s+1}^{2}\|v\|_{s-1}(\|v\|_{L^{\infty}}\|u_{j}\|_{s}+\|\nabla u_{j}\|_{L^{\infty}}\|v\|_{s-1})\\&&\qquad\qquad\qquad\qquad\quad+C \|u_{j}\|_{s+1}^{2}\|v\|_{s-1}(\|\nabla u_{k}\|_{L^{\infty}}\|v\|_{s-1}+\|\nabla v\|_{L^{\infty}}\|u_{k}\|_{s-1})\\
&&\quad\qquad\qquad\qquad\qquad\leq C\|u_{j}\|_{s+1}^{2}\|v\|_{s-1}^{2}(\|u_{j}\|_{s}+\|u_{k}\|_{s}+\|u_{k}\|_{s-1}^{2})+\|u_{j}\|_{s+1}^{2}\|v\|_{s}^{2}.
\end{eqnarray*}
and
\begin{eqnarray}\label{4.10}
&&~\left|-2\|u_{j}\|_{s+1}^{2}\sum_{|\alpha|\leq s-1}I_1\right|\nonumber\\ &&\leq C\|u_{j}\|_{s+1}^{2}\|\eta\|_{{s-1}}\|v\cdot\nabla \theta_{j}\|_{s-1}+2\|u_{j}\|_{s+1}^{2}\sum_{|\alpha|\leq s-1}|( \partial^{\alpha}\eta, \partial^{\alpha}(u_{k}\cdot \nabla\eta) )|\nonumber\\
&&\leq C\|u_{j}\|_{s+1}^{2}(\|\eta\|_{{s-1}}^{2}\nonumber\\&&\quad+\|v\|_{s-1}^{2})\|\theta_{j}\|_{s}+2\|u_{j}\|_{s+1}^{2}\sum_{|\alpha|\leq s-1}|( \partial^{\alpha}\eta, \partial^{\alpha}(u_{k}\cdot \nabla\eta) )|.
\end{eqnarray}
 Note that the second term on the right-hand side of (4.10), cannot be estimated by Lemma \ref{2.1}, the more delicate estimates are required. Using the H\"{o}lder inequality and Gagliardo-Nirenberg inequality, the term $\|\partial^{\alpha}u_{k}\cdot \nabla \eta\|$ can be controlled as follows,
\begin{eqnarray}\label{4.11}
\sum_{|\alpha|\leq s-1}\|\partial^{\alpha}u_{k}\cdot \nabla \eta\|_{L^{2}}\leq \left\{\begin{array}{ll}
C\|u_{k}\|_{H^{s-1,6}}\|\nabla \eta\|_{L^{3}}\leq C \|u_{k}\|_{s}\|\eta\|_{s-1}, ~d=3,\\
C \|u_{k}\|_{H^{s-1,4}}\|\nabla v\|_{L^{4}}\leq C \|u_{k}\|_{s}\|\eta\|_{s-1}, ~d=2,
\end{array}\right.
\end{eqnarray}
which together with (\ref{4.10}), to obtain,
\begin{equation*}
\left|-2\|u_{j}\|_{s+1}^{2}\sum_{|\alpha|\leq s-1}I_1\right|\leq C\|u_{j}\|_{s+1}^{2}(\|\eta\|_{{s-1}}^{2}+\|v\|_{s-1}^{2})(\|\theta_{j}\|_{s}+\|u_{k}\|_{s}).
\end{equation*}
Using the similar argument to (3.21) to estimate the stochastic term,
\begin{eqnarray*}
&&\quad 2\mathbb{E}\left(\sup_{r\in [0,\tau]}\left|\int_{0}^{r}\sum_{|\alpha|\leq s-1}\|u_{j}\|_{s+1}^{2}J_4 d\mathcal{W}\right|\right)\\
&&\leq C\mathbb{E}\left(\int_{0}^{\tau}\|u_{j}\|_{s+1}^{4}\left(\sum_{|\alpha|\leq s-1}J_4\right)^{2} dt\right)^{\frac{1}{2}}\\
&&\leq C\mathbb{E}\left(\int_{0}^{\tau}\|v\|_{{s-1}}^{2}\|u_{j}\|_{s+1}^{2}(\|u_{j}\|_{s+1}^{2}\|v\|_{{s-1}}^{2}
+\|u_{j}\|_{s+1}^{2}\|\eta\|_{{s-1}}^{2})dt\right)^{\frac{1}{2}}\\
&&\leq \frac{1}{2}\mathbb{E}\left(\sup_{t\in [0,\tau]}\|u_{j}\|_{s+1}^{2}\|v\|_{{s-1}}^{2}\right)+C\mathbb{E}\int_{0}^{\tau}(\|u_{j}\|_{s+1}^{2}\|v\|_{{s-1}}^{2}+\|u_{j}\|_{s+1}^{2}\|\eta\|_{{s-1}}^{2})dt,
\end{eqnarray*}
and
\begin{eqnarray*}
&&\quad 2\mathbb{E}\left(\sup_{r\in [0,\tau]}\left|\int_{0}^{r}\sum_{|\alpha|\leq s+1}\|\eta, v\|_{{s-1}}^{2}\mathcal{J}_4 d\mathcal{W}\right|\right)\\&&\leq \frac{1}{2}\mathbb{E}\left(\sup_{t\in [0,\tau]}\|u_{j}\|_{s+1}^{2}\|\eta, v\|_{{s-1}}^{2}\right)\\&&\quad+C\mathbb{E}\int_{0}^{\tau}(\|u_{j}\|_{s+1}^{2}\|\eta, v\|_{{s-1}}^{2}
+\|\theta_{j}\|_{s+1}^{2}\|\eta, v\|_{{s-1}}^{2}+\|\eta, v\|_{s-1}^{2})dt.
\end{eqnarray*}

Again, using the It\^{o} product formula to $d\|v\|_{{s-1}}^{2}\|\theta_{j}\|_{s+1}^{2}$ and $d\|\eta\|_{{s-1}}^{2}\|\theta_{j}\|_{s+1}^{2}$, we obtain,
\begin{eqnarray}\label{4.12}
&&\quad d\|v\|_{{s-1}}^{2}\|\theta_{j}\|_{s+1}^{2}=\|v\|_{{s-1}}^{2}d\|\theta_{j}\|_{s+1}^{2}+\|\theta_{j}\|_{s+1}^{2}d\|v\|_{{s-1}}^{2}
+d\|v\|_{{s-1}}^{2}d\|\theta_{j}\|_{s+1}^{2}\nonumber\\
&&\qquad\qquad\qquad\qquad=-2\|\theta_{j}\|_{s+1}^{2}\|v\|_{s}^{2}dt-2\|v\|_{{s-1}}^{2}\sum_{|\alpha|\leq s+1}\mathcal{I}_1 dt\nonumber\\&&\qquad\qquad\qquad\qquad\quad+2\|\theta_{j}\|_{s+1}^{2}\sum_{|\alpha|\leq s-1}(J_1+J_2+J_3)dt+J_{4}d\mathcal{W},
\end{eqnarray}
and
\begin{eqnarray}\label{4.13}
&&\quad d\|\eta\|_{{s-1}}^{2}\|\theta_{j}\|_{s+1}^{2}=\|\theta_{j}\|_{s+1}^{2}d\|\eta\|_{{s-1}}^{2}+\|\eta\|_{{s-1}}^{2}d\|\theta_{j}\|_{s+1}^{2}\nonumber\\
&&\qquad\qquad\qquad\qquad=-2\|\theta_{j}\|_{s+1}^{2}\sum_{|\alpha|\leq s-1}I_1 dt-2\|\eta\|_{{s-1}}^{2}\sum_{|\alpha|\leq s+1}\mathcal{I}_1 dt.
\end{eqnarray}
We mainly focus on the nonlinear terms in (\ref{4.12}) and (\ref{4.13}), and the rest of the terms are standard, so we omit them. By Lemma \ref{lem2.2}, we have,
\begin{eqnarray*}
&&\quad\left|-2\|v\|_{{s-1}}^{2}\sum_{|\alpha|\leq s+1}\mathcal{I}_1+2\|\theta_{j}\|_{s+1}^{2}\sum_{|\alpha|\leq s-1}J_1\right|\\ &&\leq C\|v\|_{{s-1}}^{2}(\|\theta_{j}\|_{s+1}^{2}+\|u_{j}\|_{s+1}^{2})\|\theta_{j}\|_{s}\\
&&\quad+C\|\theta_{j}\|_{s+1}^{2}\|v\|_{{s-1}}(\|v\|_{L^{\infty}}\|u_{j}\|_{s}+\|\nabla u_{j}\|_{L^{\infty}}\|v\|_{s-1})\\
&&\quad+C\|\theta_{j}\|_{s+1}^{2}\|v\|_{{s-1}}(\|v\|_{s-1}\|\nabla \theta_{k}\|_{L^{\infty}}+\|\nabla v\|_{L^{\infty}}\|u_{k}\|_{s-1})\\
&&\leq C \|v\|_{{s-1}}^{2}(\|\theta_{j}\|_{s+1}^{2}+\|u_{j}\|_{s+1}^{2})(\|\theta_{j}\|_{s}+\|u_{j}\|_{s}+\|\theta_{k}\|_{s}+\|u_{k}\|_{s-1}^{2})\\
&&\quad+\|\theta_{j}\|_{s+1}^{2}\|v\|_{s}^{2}.
\end{eqnarray*}
Using (\ref{4.11}) and the H\"{o}lder inequality,
\begin{eqnarray*}
&&\quad\left|-2\|\theta_{j}\|_{s+1}^{2}\sum_{|\alpha|\leq s-1}I_1-2\|\eta\|_{{s-1}}^{2}\sum_{|\alpha|\leq s+1}\mathcal{I}_1 \right|\\ &&\leq C\|\eta\|_{{s-1}}^{2}(\|\theta_{j}\|_{s+1}^{2}+\|u_{j}\|_{s+1}^{2})(\|\theta_{j}\|_{s}+\|u_{j}\|_{s})\\
&&\quad+C( \|\eta\|_{{s-1}}^{2}+\|v\|_{s-1}^{2})\|\theta_{j}\|_{s+1}^{2}\|\theta_{j}\|_{s}+C\|\theta_{j}\|_{s+1}^{2} \|\eta\|_{{s-1}}^{2}\|u_{k}\|_{s}.
\end{eqnarray*}
Combining the estimates above and using the definition of $\tau_{j,k}^{T}$ give,
\begin{eqnarray*}
&&\quad \mathbb{E}\left(\sup_{s\in [0, \tau_{j,k}^{T}\wedge t]}(\|v\|_{{s-1}}^{2}+\|\eta\|_{s-1}^{2})(\|u_{j}\|_{s+1}^{2}+\|\theta_{j}\|_{s+1}^{2})\right)\\
&&\quad+\mathbb{E}\int_{0}^{\tau_{j,k}^{T}\wedge t}\|v\|_{s}^{2}(\|\theta_{j}\|_{s+1}^{2}+\|\theta_{j}\|_{s+1}^{2})+(\|v\|_{{s-1}}^{2}+\|\eta\|_{s-1}^{2})\|u_{j}\|_{s+2}^{2}dr\\
&&\leq \mathbb{E}(\|v_{0}\|_{{s-1}}^{2}+\|\eta_{0}\|_{s-1}^{2})(\|u_{0}^{j}\|_{s+1}^{2}+\|\theta_{0}^{j}\|_{s+1}^{2})+C \mathbb{E}\int_{0}^{t}\sup_{\xi\in [0, \tau_{j,k}^{T}\wedge r]}(\|v, \eta\|_{s-1}^{2})dr\\
&&\quad+ C \mathbb{E}\int_{0}^{t}\sup_{\xi\in [0, \tau_{j,k}^{T}\wedge r]}(\|v\|_{{s-1}}^{2}+\|\eta\|_{s-1}^{2})(\|u_{j}\|_{s+1}^{2}+\|\theta_{j}\|_{s+1}^{2})dr,
\end{eqnarray*}
for any $t>0$. Thus, by applying the Gronwall inequality again, we conclude that,
\begin{eqnarray}\label{4.14}
&&\quad \mathbb{E}\left(\sup_{r\in [0, \tau_{j,k}^{T}\wedge t]}(\|v\|_{{s-1}}^{2}+\|\eta\|_{s-1}^{2})(\|u_{j}\|_{s+1}^{2}+\|\theta_{j}\|_{s+1}^{2})\right)\nonumber\\&&\leq C \mathbb{E}(\|v_{0}\|_{{s-1}}^{2}+\|\eta_{0}\|_{s-1}^{2})(\|u_{0}^{j}\|_{s+1}^{2}+\|\theta_{0}^{j}\|_{s+1}^{2})\nonumber\\&&\quad+C \mathbb{E}\left(\sup_{r\in [0, \tau_{j,k}^{T}\wedge t]}\|v, \eta\|_{s-1}^{2}\right),
\end{eqnarray}
where constant $C$ is independent of $j,k$. By the dominated convergence theorem and Lemma \ref{lem4.1} (ii) and (iii)  we obtain,
\begin{eqnarray*}
&&\quad \lim_{j\rightarrow \infty}\sup_{k\geq j}\mathbb{E}(\|v_{0}\|_{{s-1}}^{2}+\|\eta_{0}\|_{s-1}^{2})(\|u_{0}^{j}\|_{s+1}^{2}+\|\theta_{0}^{j}\|_{s+1}^{2})\nonumber\\
&&\leq C\lim_{j\rightarrow \infty}\sup_{k\geq j}\mathbb{E}(j^{2}(\|v_{0}\|_{{s-1}}^{2}+\|\eta_{0}\|_{s-1}^{2}))(\|u_{0}\|_{s}^{2}+\|\theta_{0}\|_{s}^{2})=0.
\end{eqnarray*}
For the second term on the right-hand side of \eqref{4.14}, we refer back to the estimates above. By these estimates, the Gronwall inequality and the properties of the smooth operators $\rho_{\varepsilon}$, we may infer that,
\begin{equation*}
\lim_{j\rightarrow \infty}\sup_{k\geq j}\mathbb{E}\left(\sup_{r\in [0, \tau_{j,k}^{T}\wedge t]}\|v, \eta\|_{s-1}^{2}\right)=0.
\end{equation*}
We have now established (\ref{4.7}) and hence condition (\ref{4.2}) follows.

Next, we focus on the second condition (\ref{4.3}) required by Lemma \ref{4.2}. By the It\^{o} formula,
\begin{eqnarray*}
&&\quad\sup_{t\in[0,\tau_{j}^{T}\wedge S]}\|u_{j},\theta_{j}\|_{s}^{2}+2\int_{0}^{\tau_{j}^{T}\wedge S}\|u_{j}\|_{s+1}^{2}dt\\
&&\leq\|u_{0}^{j},\theta_{0}^{j}\|_{s}^{2}+\sum_{|\alpha|\leq s}\int_{0}^{\tau_{j}^{T}\wedge S}|\mathcal{J}_1+\mathcal{J}_2+\mathcal{J}_3+\mathcal{I}_1|dt+\sup_{t\in[0,\tau_{j}^{T}\wedge S]}\left|\int_{0}^{t}\sum_{|\alpha|\leq s}\mathcal{J}_4 d\mathcal{W}\right|,
\end{eqnarray*}
leading to,
\begin{eqnarray}\label{4.15}
&&\quad\mathbb{P}\left(\sup_{t\in[0,\tau_{j}^{T}\wedge S]}\|u_{j},\theta_{j}\|_{m}^{2}+2\int_{0}^{\tau_{j}^{T}\wedge S}\|u_{j}\|_{m+1}^{2}dt>\|u_{0}^{j},\theta_{0}^{j}\|_{s}^{2}+1\right)\nonumber\\
&&\leq \mathbb{P}\left(\sum_{|\alpha|\leq s}\int_{0}^{\tau_{j}^{T}\wedge S}|\mathcal{J}_1+\mathcal{J}_2+\mathcal{J}_3+\mathcal{I}_1|dt> \frac{1}{2}\right)\nonumber\\
&&\quad+\mathbb{P}\left(\sup_{t\in[0,\tau_{j}^{T}\wedge S]}\left|\int_{0}^{t}\sum_{|\alpha|\leq s}\mathcal{J}_4d\mathcal{W}\right|>\frac{1}{2}\right).
\end{eqnarray}
For the first term on the right-hand side of \eqref{4.15}, applying the Chebyshev inequality and Lemma \ref{lem2.2} give,
\begin{eqnarray}\label{4.16}
&&\quad \mathbb{P}\left(\sum_{|\alpha|\leq s}\int_{0}^{\tau_{j}^{T}\wedge S}|\mathcal{J}_1+\mathcal{J}_2+\mathcal{J}_3+\mathcal{I}_1|dt> \frac{1}{2}\right)\nonumber\\
&&\leq \mathbb{E}\int_{0}^{\tau_{j}^{T}\wedge S}(\|\theta_{j}\|_{s}^{2}+\|u_{j}\|_{s}^{2})(1+\|\theta_{j}\|_{s}+\|u_{j}\|_{s})dt\leq CS,
\end{eqnarray}
where the constant $C$ is independent of $k$ and $S$. For second term on the right hand in (\ref{4.15}), applying Doob's inequality and the It\^{o} isometry formula, we obtain,
\begin{eqnarray}\label{4.17}
&&\quad\mathbb{P}\left(\sup_{t\in[0,\tau_{j}^{T}\wedge S]}\left|\int_{0}^{t}\sum_{|\alpha|\leq s}\mathcal{J}_4d\mathcal{W}\right|>\frac{1}{2}\right)\nonumber\\
&&\leq C\mathbb{E}\left(\int_{0}^{\tau_{j}^{T}\wedge S}\sum_{|\alpha|\leq s}\mathcal{J}_4d\mathcal{W}\right)^{2} \leq C\mathbb{E}\int_{0}^{\tau_{j}^{T}\wedge S}\sum_{|\alpha|\leq s}\mathcal{J}_4^{2}dt\nonumber\\
&&\leq C\mathbb{E}\int_{0}^{\tau_{j}^{T}\wedge S}\|u_{j}\|_{s}^{2}(1+\|u_{j}\|_{s}^{2}+\|\theta_{j}\|_{s}^{2})dt\leq CS,
\end{eqnarray}
where the constant $C$ is independent of $k$ and $S$. Combining (\ref{4.16}) and (\ref{4.17}), the proof of condition (\ref{4.3}) is now complete.
\end{proof}
Both conditions (\ref{4.2}) and (\ref{4.3}) have been established, following Lemma {\ref{lem4.2}}, we thus obtain the desired results of strong convergence $\mathbb{P}$ a.s. and the uniform bound of the approximate solutions. Hence, we can show that $(u,\theta,\tau)$ is a local pathwise solution in the sense of Definition \ref{def2.1} using the same argument as \cite{Breit}. We have imposed the bound on the initial data $(u_{0},\theta_{0})$ in order to apply Lemma \ref{lem4.2}, which can be removed as mentioned in Section 3 using a cutting argument, and then extend the local solution to the maximal pathwise solution as shown in subsection 3.4 via maximality arguments. The proof of Theorem \ref{the2.1} is now complete.

\section{The global existence for the case of additive noise}\label{sec5}
\setcounter{equation}{0}

In this section, we shall establish the global existence of strong pathwise solutions to 2D system \eqref{Equ1.1} driven by an additive noise with large initial data. We remark that the local existence of such a pathwise solution can be obtained by a more direct approach given in \cite{Kim1} where the local existence of pathwise solution was proved for the stochastic Euler equations with additive noise using a change of variable to transform the stochastic PDE to a random PDE such that the deterministic result can be applied.

 We give a proposition offering a criterion for the global existence of solution.

\begin{proposition}\label{pro5.1} Fix a stochastic basis $\mathcal{S}:=(\Omega,\mathcal{F}, \{\mathcal{F}_{t}\}_{t\geq 0}, \mathbb{P}, \mathcal{ W})$. If the triple $(u,\theta, \xi)$ is a unique maximal pathwise solution. Define the stopping time $\tau_{R}$ as follows,
\begin{equation*}
\tau_{R}:=\inf\left\{T\geq 0:\sup_{t\in[0,\xi\wedge T]}\|\nabla u,\nabla\theta\|_{L^{\infty}}>R\right\}.
\end{equation*}
Then, for any $T, R>0$,
\begin{equation}\label{5.1}
\mathbb{E}\left(\sup_{t\in[0,T\wedge \xi\wedge \tau_{R}]}\|u,\theta\|_{s}^{2}\right)+\mathbb{E}\int_{0}^{T\wedge \xi\wedge \tau_{R}}\|u\|_{s+1}^{2}dt<\infty,
\end{equation}
and $\tau_{R}\leq \xi$ a.s. Furthermore, if $\lim_{R\rightarrow \infty}\tau_{R}=\infty$, then $(u,\theta)$ is a global solution in the sense of Definition \ref{def2.2}.
\end{proposition}
\begin{proof}
 Applying the It\^{o} formula to function $\|u,\theta\|_{s}^{2}$, then using Lemma \ref{lem2.2} and the H\"{o}lder inequality, we have,
\begin{align}\label{5.2}
d\|u,\theta\|_{s}^{2}+\|u\|_{s+1}^{2}dt&\leq C(1+\|\nabla u\|_{L^{\infty}}+\|\nabla \theta\|_{L^{\infty}})\|u,\theta\|_{s}^{2}dt+ C\|f\|_{L_{2}(H;H^{s})}^{2}dt\nonumber\\
&\quad+2(u,Pf)_{s}d\mathcal{W},
\end{align}
for some constant $C=C(s, \mathbb{T}^{d})$. For any fixed $0\leq \tau_{a}\leq \tau_{b}< T\wedge \xi\wedge \tau_{R}$, integrating in time, taking the integral over interval $[\tau_{a},\tau_{b}]$, using $(\ref{5.2})$, and the Burkholder-Davis-Gundy inequality, we obtain,
\begin{eqnarray*}
&&\mathbb{E}\left(\sup_{t\in[\tau_{a},\tau_{b}]}\|u,\theta\|_{s}^{2}\right)
+\mathbb{E}\int_{\tau_{a}}^{\tau_{b}}\|u\|_{s+1}^{2}dt\leq \mathbb{E}\|u(\tau_{a}),\theta(\tau_{a})\|_{s}^{2}\\&&+\mathbb{E}\int_{\tau_{a}}^{\tau_{b}}(1+\|\nabla u\|_{L^{\infty}}+\|\nabla \theta\|_{L^{\infty}})\|u,\theta\|_{s}^{2}dt
+\mathbb{E}\int_{\tau_{a}}^{\tau_{b}}\|f\|_{L_{2}(H;H^{s})}^{2}dt.
\end{eqnarray*}
Then, $(\ref{5.1})$ follows from  the stochastic Gronwall lemma given in \cite{Glatt}.

Next, we show that $\tau_{R}\leq \xi$ by a contradiction argument. Suppose not, then there exists a deterministic time $T$ such that $\mathbb{P}\{\tau_{R}\wedge T>\xi\}>0$ due to the fact $\{\tau_{R}>\xi\}=\bigcup_{T\geq 0}\{\tau_{R}\wedge T>\xi\}$. By the definition of $\xi$, we infer that,
\begin{equation*}
\sup_{t\in[0,T\wedge\tau_{R}\wedge \xi]}\|u,\theta\|_{s}^{2}+\int_{0}^{T\wedge\tau_{R}\wedge \xi}\|u\|_{s+1}^{2}dt\geq \sup_{t\in[0, \xi]}\|u,\theta\|_{s}^{2}=\infty.
\end{equation*}
Since $\mathbb{P}\{\tau_{R}\wedge T>\xi\}>0$, this leads to a contradiction with (\ref{5.1}) and hence we obtain the result.
\end{proof}

Before showing $\tau_{R}\rightarrow \infty$ as $R\rightarrow \infty$, we state a condition for the nonblow-up of solutions to stochastic ODEs, which is taken from \cite{Fang,Glatt-Holtz} and modified to our setting.
\begin{lemma}\label{lem5.1} Fix a stochastic basis $\mathcal{S}:=(\Omega,\mathcal{F}, \{\mathcal{F}_{t}\}_{t\geq 0}, \mathbb{P}, \mathcal{ W})$. Suppose that on $\mathcal{S}$ we have defined $Y$ a real valued, predictable process defined up to a time $\xi> 0$, that is,  for all bounded stopping times $\tau< \xi$, $\sup_{t\in[0,\tau]}Y< \infty $ a.s. Assume that $Y\geq 1$ and it satisfies the It\^{o} stochastic differential equation
\begin{equation*}\label{5.4}
dY+\nu Y_{1}dt=Xdt+Zd\mathcal{W}, ~~~~Y(0)=Y_{0},
\end{equation*}
on $[0,\xi)$, where $Y_{1}>0$ and $\nu$ is a positive constant, $X$ is real-valued and $Z$ is an $L_{2}$-valued predictable processes. Suppose further that there exists a stochastic process,
\begin{equation*}\label{5.5}
\sigma\in L^{1}(\Omega;L^{1}_{loc}[0,\infty)),
\end{equation*}
with $\sigma\geq 1 $ for almost every $(\omega,t)$ and an increasing collection of stopping times $\Gamma_{R}$ with $\Gamma_{R}\leq \xi$ such that,
\begin{equation*}\label{5.6}
\mathbb{P}\left(\mathop{\cap}\limits_{R>0} \{\Gamma_{R}< \xi \wedge T\}\right)=0,
\end{equation*}
 and for every fixed $R>0$, there exists a process $g(t)$, a number $r\in [0,\frac{1}{2}]$, and a constant $C$ such that,
\begin{equation*}
|X|\leq \frac{\nu}{2}Y_{1}+C(g(t)\cdot(1+\log Y)Y+\sigma), ~~~\|Z\|_{L_{2}}\leq CY^{1-r}\sigma^{r},
\end{equation*}
where the process $g$ satisfies $\mathbb{E}\int_{0}^{t}g(s)ds<C(R,T)$, for any $t\in [0, \Gamma_{R}]$. Then we have $\sup_{t\in[0,\xi\wedge T]}Y<\infty$, a.s. for each $T>0$.
\end{lemma}

We next establish the condition which can be used to obtain $\tau_{R}\rightarrow \infty$ as $R\rightarrow \infty$ by the Lemma \ref{lem5.1}.
\begin{proposition}\label{pro5.2} Fix $s>2$, and assume that $f$ satisfies condition (\ref{2.10}). If $(u,\theta,\xi)$ is the maximal pathwise solution, then,
\begin{equation}\label{5.8}
\sup_{t\in[0,T\wedge \xi]}\|\nabla u, \nabla\theta\|_{L^{\infty}}< \infty, \mbox{a.s.}
\end{equation}
for each $T>0$.
\end{proposition}
\begin{proof} In order to obtain the suitable estimates, let $w=\nabla^{\bot}\cdot u$ and $\eta=\nabla^{\bot}\theta$, where $\nabla^{\bot}=(-\partial_{2},\partial_{1})$ and then take the operator $\nabla^{\bot}$ on both sides of the system, yields,
\begin{eqnarray}\label{5.9}
\left\{\begin{array}{ll}
dw-\triangle w+u\cdot \nabla w=-\theta_{x_{1}}+ \nabla^{\bot}\cdot fd\mathcal{W},\\
d\eta+u\cdot \nabla\eta=\eta\cdot \nabla u.
\end{array}\right.
\end{eqnarray}
Note that, comparing to the three dimensional case, there is no vortex stretching term $w\cdot \nabla u$ appearing in (\ref{5.9}), which makes the global existence achievable. Here, different from the Euler equaiton, the diffusion term plays a key role in the later estimates due to the coupled construction. Multiplying the second equation in (\ref{5.9}) by $\eta|\eta|^{p-2}$ and integrating over $\mathbb{T}^{2}$ we obtain
\begin{equation*}
\frac{1}{p}\frac{d}{dt}\|\nabla\theta\|_{L^{p}}^{p}\leq \|\nabla u\|_{L^{\infty}}\|\nabla\theta\|_{L^{p}}^{p},
\end{equation*}
where we have used the cancellation property $( u\cdot \nabla v,v|v|^{p-2})=0$. Integrating with respect to time and letting $p\rightarrow\infty$, the above estimates give,
\begin{equation*}
\|\nabla\theta\|_{L^{\infty}}\leq \|\nabla\theta_{0}\|_{L^{\infty}}{\rm exp}\int_{0}^{t}\|\nabla u\|_{L^{\infty}}ds.
\end{equation*}
Moreover, the Sobolev embedding theorem and the Biot-Savart law, yield,
\begin{equation*}
\|\nabla u\|_{L^{\infty}}\leq C \|\nabla u\|_{L^{2}}+C\|\nabla w\|_{L^{4}},
\end{equation*}
where $C=C(\mathbb{T}^{2})$ is a positive constant. Therefore, the proof will be completed once we obtain the bound for $\|w\|_{L^{p}}$ for $p\geq 2$ and $\|\nabla w\|_{L^{4}}$. From the temperature equation, we easily have,
\begin{equation}\label{5.10}
\|\theta\|_{L^{p}}\leq \|\theta_{0}\|_{L^{p}}, ~~{\rm \forall} t\in [0,T],~~p\in[1,\infty].
\end{equation}
After applying the It\^{o} formula to the function $\|w\|_{L^{p}}^{p}$ for $p\geq 2$, and integrating by parts, we arrive at
\begin{eqnarray}\label{5.11}
&&\quad d\|w\|_{L^{p}}^{p}+p(p-1)\int_{\mathbb{T}^{2}}|\nabla w|^{2}|w|^{p-2}dxdt\nonumber\\
&&=-p( u\cdot \nabla w, w|w|^{p-2}) dt-p(\theta_{x_{1}}, w|w|^{p-2}) dt+\frac{p}{2}\int_{\mathbb{T}^{2}}|w|^{p-2}|\nabla^{\bot}\cdot f|^{2}dxdt\nonumber\\
&&\quad+\frac{p(p-2)}{2}\int_{\mathbb{T}^{2}}|w|^{p-4}(w\cdot(\nabla^{\bot}\cdot f))^{2}dxdt+p( w|w|^{p-2}, \nabla^{\bot}\cdot f) d\mathcal{W}\nonumber\\
&&\leq \frac{p(p-1)}{2}\int_{\mathbb{T}^{2}}|\nabla w|^{2}|w|^{p-2}dxdt+\frac{p(p-1)}{2}\int_{\mathbb{T}^{2}}\theta^{2}|w|^{p-2}dxdt\nonumber\\
&&\quad+\|w\|_{L^{p}}^{p-2}\|\nabla^{\bot}\cdot f\|_{\mathbb{W}^{0,p}}^{2}dt+p( w|w|^{p-2}, \nabla^{\bot}\cdot f) d\mathcal{W}.
\end{eqnarray}
Define the stopping time $\tau_{R}$ by
\begin{equation*}
\tau_{R}=\inf\{t\geq 0:\|w\|_{L^{p}}>R\}\wedge \xi.
\end{equation*}
From the definition of $\xi$ as the maximal time of existence of solution, it follows that $\tau_{R}\rightarrow \xi$ a.s. as $R\rightarrow \infty$. In addition, for every $T\geq 0$ and a.s. $\omega$, if $R$ is sufficiently large we have $\tau_{R}\wedge T=\xi \wedge T$.
For the stochastic term, using the Burkholder-Davis-Gundy inequality (\ref{2.4}),
\begin{eqnarray*}
&&\quad \mathbb{E}\left(\sup_{s\in[0,\tau_{R}\wedge T]}\left|\int_{0}^{s}( w|w|^{p-2}, \nabla^{\bot}\cdot f) d\mathcal{W}\right|\right)\\&& \leq \frac{1}{2}\mathbb{E}\left(\sup_{s\in[0,\tau_{R}\wedge T]}\|w\|_{L^{p}}^{p}\right)+C\mathbb{E}\int_{0}^{\tau_{R}\wedge T}\|w\|_{L^{p}}^{p}dt+C\mathbb{E}\int_{0}^{\tau_{R}\wedge T}\|\nabla^{\bot}\cdot f\|_{\mathbb{W}^{0,p}}^{p}dt.
\end{eqnarray*}
Taking the integral in time then taking the expectation in (\ref{5.11}), and using conditions (\ref{2.10}) and (\ref{5.10}), the H\"{o}lder inequality and the Gronwall inequality, we obtain,
\begin{eqnarray*}
&&\quad \mathbb{E}\left(\sup_{s\in[0,\tau_{R}\wedge T]}\|w\|_{L^{p}}^{p}\right)+\mathbb{E}\int_{0}^{\tau_{R}\wedge T}\int_{\mathbb{T}^{2}}|\nabla w|^{2}|w|^{p-2}dxdt\\
&&\leq C\mathbb{E}\|w_{0}\|_{L^{p}}^{p}+\mathbb{E}\int_{0}^{\tau_{R}\wedge T}(\|\theta\|_{L^{p}}^{p}+\|\nabla^{\bot}\cdot f\|_{\mathbb{W}^{0,p}}^{p})dt\leq C,
\end{eqnarray*}
where constant $C=C(\mathbb{T}^{2},T, \mathbb{E}\|w_{0},\theta_{0}\|_{L^{p}}^{p},p)$ is independent of $R$. We conclude that for all $R>0, p\geq 2$, $\sup_{t\in [0,\tau_{R}\wedge T]}\|w\|_{L^{p}}^{p}+\int_{0}^{\tau_{R}\wedge T}\int_{\mathbb{T}^{2}}|\nabla w|^{2}|w|^{p-2}dxdt<\infty$ a.s. Then we finally conclude that for a.s. $\omega$,
\begin{equation}\label{5.12}
\sup_{t\in [0,\xi\wedge T]}\|w\|_{L^{p}}^{p}+\int_{0}^{\xi\wedge T}\int_{\mathbb{T}^{2}}|\nabla w|^{2}|w|^{p-2}dxdt<\infty.
\end{equation}
Next, taking the operation $\nabla$ on the first equation in (\ref{5.9}), and applying the It\^{o} formula to the function $|\nabla w|^{4}$, then integrating by parts we obtain,
\begin{eqnarray*}
&&\quad d\|\nabla w\|_{L^{4}}^{4}+12\int_{\mathbb{T}^{2}}|\nabla^{2}w|^{2}|\nabla w|^{2}dxdt\\
&&=-4(\nabla(u\cdot \nabla u),\nabla w|\nabla w|^{2}) dt-4(\nabla\theta_{x_{1}},\nabla w|\nabla w|^{2}) dt
+2\int_{\mathbb{T}^{2}}|\nabla w|^{2}|\nabla\nabla^{\bot}\cdot f|^{2}dxdt\\&&\quad+4\int_{\mathbb{T}^{2}}(\nabla w\cdot (\nabla\nabla^{\bot}\cdot f))^{2}dxdt+4(\nabla \nabla^{\bot}\cdot f, \nabla w|\nabla w|^{2}) d\mathcal{W}\\
&&=(I_{1}+I_{2}+I_{3}+I_{4})dt+I_{5}d\mathcal{W}.
\end{eqnarray*}
For $I_{1}$ and $I_{2}$, after integration by parts and applying the Young inequality, we have,
\begin{eqnarray}\label{5.13}
&&~|I_{1}|\leq 3\int_{\mathbb{T}^{2}}|\nabla^{2}w|^{2}|\nabla w|^{2}dx+C\int_{\mathbb{T}^{2}}|u|^{2}|\nabla w|^{4}dx\nonumber\\&&\qquad\leq 3\int_{\mathbb{T}^{2}}|\nabla^{2}w|^{2}|\nabla w|^{2}dx+C\|u\|_{L^{\infty}}^{2}\|\nabla w\|_{L^{4}}^{4},
\end{eqnarray}
and
\begin{eqnarray}\label{5.14}
&&~|I_{2}|\leq 3\int_{\mathbb{T}^{2}}|\nabla^{2}w|^{2}|\nabla w|^{2}dx+C\int_{\mathbb{T}^{2}}|\nabla \theta|^{2}|\nabla w|^{2}dx\nonumber\\&&\qquad\leq
3\int_{\mathbb{T}^{2}}|\nabla^{2}w|^{2}|\nabla w|^{2}dx+C(\|\nabla w\|_{L^{4}}^{4}+\|\nabla\theta\|_{L^{4}}^{4}).
\end{eqnarray}
For $I_{3}$ and $I_{4}$, we can easily   get,
\begin{equation}\label{5.15}
|I_{3}+I_{4}|\leq 3(\|\nabla w\|_{L^{4}}^{4}+\|\nabla\nabla^{\bot}\cdot f\|_{\mathbb{W}^{0,4}}^{4}).
\end{equation}

In order to close the estimates, we also need a bound for $\|\nabla\theta\|_{L^{4}}^{4}$. Taking the inner product with $\eta|\eta|^{2}$, we deduce,
\begin{align}\label{5.16}
d\|\nabla\theta\|_{L^{4}}^{4}&=4(\nabla^{\bot} \theta|\nabla^{\bot} \theta|^{2},\nabla^{\bot} \theta\cdot \nabla u)dt\leq 4\|\nabla u\|_{L^{\infty}}\|\nabla\theta\|_{L^{4}}^{4}dt\nonumber\\
&\leq C(1+\|\nabla u\|_{L^{2}}+\|\nabla^{2}u\|_{L^{2}})(1+\log^{+}\|\nabla^{2}u\|_{L^{4}})\|\nabla \theta\|_{L^{4}}^{4}dt\nonumber\\
&\leq C(1+\|w\|_{L^{2}}+\|\nabla w\|_{L^{2}})(1+\log^{+}(\|\nabla w\|_{L^{4}}^{4}+\|\nabla \theta\|_{L^{4}}^{4}))\|\nabla \theta\|_{L^{4}}^{4}dt,
\end{align}
where again we have used the cancellation property $(u\cdot \nabla v,v|v|^{2})=0$ and applied the following form of the Brezis-Wainger inequality \cite{Engler},
\begin{equation*}
\|h\|_{L^{\infty}}\leq C(1+\|\nabla h\|_{L^{2}})(1+\log^{+}\|\nabla h\|_{L^{p}})^{\frac{1}{2}}+C\|h\|_{L^{2}},
\end{equation*}
for $h\in L^{2}(\mathbb{T}^{2})\cap H^{1,p}(\mathbb{T}^{2})$, which holds for $p> 2$.
Combining (\ref{5.13})-(\ref{5.16}), we have,
\begin{equation}\label{5.17}
X \leq 6\int_{\mathbb{T}^{2}}|\nabla^{2}w|^{2}|\nabla w|^{2}dx+Cg(t)(1+\log Y)Y+\sigma,
\end{equation}
where  $Y=1+\|\nabla w\|_{L^{4}}^{4}+\|\nabla\theta\|_{L^{4}}^{4}$, $g(t)=1+\|u\|_{L^{\infty}}^{2}+\|w\|_{L^{2}}+\|\nabla w\|_{L^{2}}$ and $\sigma=(1+\|\nabla\nabla^{\bot}\cdot f\|_{\mathbb{W}^{0,4}})^{4}$. For $Z=I_{5}=( \nabla\nabla^{\bot}\cdot f, \nabla w|\nabla w|^{2}) $, note that,
\begin{eqnarray}\label{5.18}
&&\|Z\|_{L_{2}}\leq |(\nabla\nabla^{\bot}\cdot f, \nabla w|\nabla w|^{2})|\leq \| \nabla\nabla^{\bot}\cdot f\|_{\mathbb{W}^{0,4}}\|\nabla w\|_{L^{4}}^{3}\nonumber\\&&\qquad\quad \leq (1+\| \nabla\nabla^{\bot}\cdot f\|_{\mathbb{W}^{0,4}})Y^{\frac{3}{4}}.
\end{eqnarray}
Define the stopping time $\Gamma_{R}$ by,
\begin{equation}\label{5.19}
\Gamma_{R}=\inf\left\{t\geq 0:\|w\|_{L^{2}}+\|w\|_{L^{4}}+\int_{0}^{t}\|\nabla w\|_{L^{2}}ds>R\right\}\wedge \xi.
\end{equation}
Obviously, $\Gamma_{R}$ is increasing in $R$ and $\mathbb{P}\left( \cap_{R}\{\Gamma_{R}< \xi
\wedge T\}\right)=0$ due to estimate (\ref{5.12}) for $p=2,4$. With (\ref{5.17})-(\ref{5.19}) established, $\sup_{t\in[0,\xi\wedge T]}(\|\nabla w\|_{L^{4}}^{4}+\|\nabla\theta\|_{L^{4}}^{4})<\infty$ follows from Lemma \ref{lem5.1} for every $T>0$. Then we have $\sup_{t\in[0,\xi\wedge T]}\|\nabla u,\nabla \theta\|_{L^{\infty}}<\infty$,
for each $T>0$, completing the proof.
\end{proof}

With Proposition \ref{pro5.2}, according to the definition of $\tau_{R}$ in Proposition 5.1, we have $\lim_{R\rightarrow \infty}\tau_{R}=\infty$. Then $\xi=\infty$, that is, $(u,\theta)$ is a global pathwise solution in the sense of Definition \ref{def2.2}. This completes the proof of Theorem \ref{the2.2}.

\bigskip

\section{Large deviation principle}\label{sec6}
\setcounter{equation}{0}

With the global existence and uniqueness of solution achieved in the previous section, we consider the large deviation principle via the weak convergence approach. Since there is no diffusion term in the temperature equation, the weak convergence is proved by a compactness argument, and we are only able to prove the large deviation principle in the nonoptimal space $\mathcal{X}$ which will be introduced later.

Define the class $\mathcal{A}$ as the set of $H_{0}$-valued predictable stochastic processes $h$ such that $\int_{0}^{T}\|h\|_{0}^{2}dt<\infty$ a.s. For any fixed $M>0$, recall the set
\begin{equation*}
S_{M}=\left\{h\in L^{2}(0,T;H_{0}):\int_{0}^{T}\|h\|_{0}^{2}dt\leq M\right\}.
\end{equation*}
The set $S_{M}$,  endowed with the weak topology $d(h,g)=\sum_{k\geq 1}\frac{1}{2^{k}}\left|\int_{0}^{T}\langle h(t)-g(t), e_{k}\rangle_{0}dt\right|$ with $\{e_{k}\}_{k\geq 1}$ being an orthonormal basis of $L^{2}(0,T;H_{0})$, is a Polish space. For $M>0$, define $\mathcal{A}_{M}=\{h\in \mathcal{A}:h(\omega)\in S_{M},a.s.\}$. For a Polish space $\mathcal{X}$, a function $I$: $\mathcal{X}\rightarrow [0,\infty]$ is called a rate function if $I$ is lower semicontinuous and is referred to as a good rate function if for each $M<\infty$, the level set $\{x\in \mathcal{X}:I(x)\leq M\}$ is compact. Then for a family $\{X^{\epsilon}\}_{\epsilon >0}$ in $\mathcal{X}$, we say that the large deviation principle (LDP) holds with rate function $I$ if the family obeys the following two conditions:\\
a. LDP lower bound: for every open set $U\subset \mathcal{X}$,
\begin{equation*}
-\inf_{x\in U} I(x) \leq \liminf_{\epsilon \rightarrow 0}\epsilon \log \mathbb{P}(X^{\epsilon} \in U),
\end{equation*}
b. LDP upper bound: for every closed set $C \subset \mathcal{X}$,
\begin{equation*}
\limsup_{\epsilon \rightarrow 0} \epsilon \log \mathbb{P}(X^{\epsilon} \in C) \leq -\inf_{x\in C}I(x).
\end{equation*}
Furthermore, $\{X^{\epsilon}\}_{\epsilon >0}$ satisfies the Laplace principle in $\mathcal{X}$ with rate function $I$ if for each real-valued, bounded and continuous function $f$, we have
\begin{equation*}
\lim_{\epsilon\rightarrow 0}\epsilon\log \mathbb{E}\left\{{\rm exp}\left[-\frac{1}{\epsilon}f(X^{\epsilon})\right]\right\}=-\inf_{x\in \mathcal{X}}\{f(x)+I(x)\}.
\end{equation*}

 For more background in this area of study we recommend \cite{Dembo, Ellis}. Since $\{X^{\epsilon}\}_{\epsilon >0}$ is a Polish space valued random process, the Laplace principle and the large deviation principle are equivalent, see \cite[Theorem 1.2.3]{Ellis}. To apply the weak convergence approach, we will use the following theorem given in \cite{Dupuis}. For examples of results on large deviations for stochastic PDEs by applying the theorem below see \cite{Duan,Millet,Sundar,Chueshov}.
\begin{theorem}{\rm \cite[Theorem 5]{Dupuis}} \label{the6.1} For Polish spaces $\mathcal{X},\mathcal{Y}$ and each $\epsilon>0$, let $\mathcal{G}^{\epsilon}:\mathcal{Y}\rightarrow \mathcal{X}$ be a measurable map and define $U^{\epsilon}:=\mathcal{G}^{\epsilon}(\sqrt{\epsilon}\mathcal{W})$ where $\mathcal{W}$ is a $Q$-Wiener process. If there is a measurable map $\mathcal{G}^{0}:\mathcal{Y}\rightarrow \mathcal{X}$ such that the following conditions hold,\\
(i) For $M<\infty$, if $h_{\epsilon}$ converges in distribution to $h$ as $S_{M}$-valued random elements, then,
\begin{eqnarray*}
\mathcal{G}^{\epsilon}\left(\sqrt{\epsilon}\mathcal{W}+\int_{0}^{\cdot}h_{\epsilon}(t)dt\right)\rightarrow \mathcal{G}^{0}\left(\int_{0}^{\cdot}hdt\right)
\end{eqnarray*}
as $\epsilon\rightarrow 0$ in distribution $\mathcal{X}$.\\
(ii) For every $M<\infty$, the set
\begin{eqnarray*}
K_{M}=\left\{\mathcal{G}^{0}\left(\int_{0}^{\cdot}hdt\right):h\in S_{M}\right\}
\end{eqnarray*}
is a compact subset of $\mathcal{X}$.
Then, family $\{U^{\epsilon}\}_{\epsilon >0}$ satisfies the large deviation principle with the rate function
\begin{eqnarray*}
I(U)=\inf_{\{h\in L^{2}(0,T;H_{0}):U=\mathcal{G}^{0}(\int_{0}^{\cdot}h(t)dt)\}}\left\{\frac{1}{2}\int_{0}^{T}\|h\|_{0}^{2}dt\right\}.
\end{eqnarray*}
\end{theorem}

We consider the Polish space,
\begin{eqnarray*}
&&\mathcal{X}=[\mathcal{C}([0,T];X^{s-1})\cap L^{2}(0,T;X^{s})]\times  \mathcal{C}([0,T];H^{s-1}),
\end{eqnarray*}
and let $\mathcal{B}(\mathcal{X})$ be the Borel $\sigma$-field of the Polish space $\mathcal{X}$.

Recall the stochastic Boussinesq equations given by,
\begin{eqnarray}\label{6.2}
\left\{\begin{array}{ll}
du^{\epsilon}+Au^{\epsilon}dt+P(u^{\epsilon}\cdot\nabla)u^{\epsilon}dt=P\theta^{\epsilon} e_{2}dt+\sqrt{\epsilon}Pf d\mathcal{W},\\
d\theta^{\epsilon}+(u^{\epsilon}\cdot\nabla)\theta^{\epsilon}dt=0,
\end{array}\right.
\end{eqnarray}
with initial data $U_{0}=(u_{0},\theta_{0})$. By the previous section, there exists a strong pathwise solution $U^{\epsilon}$ of system \eqref{6.2} with values in $[\mathcal{C}([0,T];X^{s})\cap L^{2}(0,T;X^{s+1})]\times \mathcal{C}([0,T];H^{s})$ for all $T>0$, and it is pathwise unique in $\mathcal{X}$. It follows that there exists a Borel-measurable function $\mathcal{G}^{\epsilon}:\mathcal{C}([0,T];H)\rightarrow \mathcal{X}$ such that $\mathcal{G}^{\epsilon}(\mathcal{W(\cdot)})=U^{\epsilon}(\cdot)$, $\mathbb{P}$ a.s.  We consider the large deviation principle for $\{U^{\epsilon}\}_{\epsilon >0}$ as $\epsilon\rightarrow 0$.
\begin{lemma}\label{lem6.1}For any $h\in \mathcal{A}_{M}$, let $\mathcal{G}^{\epsilon}\left(\sqrt{\epsilon}\mathcal{W}+\int_{0}^{\cdot}h(t)dt\right)$ be denoted by $U_{h}^{\epsilon}$. Then $U_{h}^{\epsilon}$ is the unique strong pathwise solution of
\begin{eqnarray}\label{6.3}
\left\{\begin{array}{ll}
du^{\epsilon}+Au^{\epsilon}dt+P(u^{\epsilon}\cdot\nabla)u^{\epsilon}dt=P\theta^{\epsilon} e_{2}dt+\sqrt{\epsilon}Pf d\mathcal{W}+Pf hdt,\\
d\theta^{\epsilon}+(u^{\epsilon}\cdot\nabla)\theta^{\epsilon}dt=0.
\end{array}\right.
\end{eqnarray}
with the initial data $U_{0}=(u_{0},\theta_{0})$.
\end{lemma}
\begin{proof} The proof can be easily achieved using the Girsanov transformation argument. For details see Theorem 10 of \cite{Dupuis} or Lemma 4.1 of \cite{Sundar}.
\end{proof}

Although we obtain the well-posedness of the stochastic controlled equation note that the Girsanov density, ${{\rm exp}(\frac{1}{\epsilon}\int_{0}^{t}\|h\|_{0}^{2}dr)}$, is not uniformly bounded in $L^{2}$ as $\epsilon\rightarrow 0$. Therefore, the uniform a priori estimates which is important in the proof of weak convergence result, cannot be deduced from the corresponding ones for stochastic Boussinesq equations.

Next, we shows that the solution
\begin{eqnarray*}
U^{\epsilon}_{h}\in [L^2(\Omega;\mathcal{C}([0,T];X^{s})\cap L^{2}(0,T;X^{s+1}))]\times L^2(\Omega;\mathcal{C}([0,T];H^{s}))
\end{eqnarray*}
of system (\ref{6.3}) is bounded uniformly in $\epsilon$. The proof is quite similar to those of Propositions \ref{pro5.1} and \ref{pro5.2}. To simplify the notation, we replace $U_{h}^{\epsilon}:=(u^{\epsilon},\theta^{\epsilon})$ by $(u,\theta)$. First, we have,
\begin{eqnarray*}
 &&\quad\mathbb{E}\left(\sup_{t\in[0,T]}\|u,\theta\|_{s}^{2}\right)+\mathbb{E}\int_{0}^{T}\|u\|_{s+1}^{2}dt\\&&\leq \mathbb{E}\|U_{0}\|_{s}^{2}
 +\mathbb{E}\int_{0}^{T}(1+\|\nabla u\|_{L^{\infty}}+\|\nabla \theta\|_{L^{\infty}})\|u,\theta\|_{s}^{2}dt\\
&&\quad+C\mathbb{E}\int_{0}^{T}\|f\|_{L_{2}(H;H^{s})}^{2}dt+\mathbb{E}\int_{0}^{T}(fh,u)_{s}dt.
\end{eqnarray*}
The H\"{o}lder inequality yields,
\begin{equation}\label{6.4}
|(fh,u)_{s}|\leq \|u\|_{s}\|f\|_{L_{2}(H,H^{s})}\|h\|_{H_{0}}\leq C(\|u\|_{s}^{2}\|h\|_{H_{0}}+\|f\|_{L_{2}(H,H^{s})}^{2}\|h\|_{H_{0}}).
\end{equation}
By (6.3), we arrive at,
\begin{eqnarray*}
&&\quad\mathbb{E}\left(\sup_{t\in[0,T]}\|u,\theta\|_{s}^{2}\right)+\mathbb{E}\int_{0}^{T}\|u\|_{s+1}^{2}dt\\
&&\leq \mathbb{E}\|U_{0}\|_{s}^{2}+\mathbb{E}\int_{0}^{T}(1+\|\nabla u\|_{L^{\infty}}+\|\nabla \theta\|_{L^{\infty}}+\|h\|_{H_{0}})\|u,\theta\|_{s}^{2}dt\\
&&\quad+C\mathbb{E}\int_{0}^{T}\|f\|_{L_{2}(H;H^{s})}^{2}(1+\|h\|_{H_{0}})dt.
\end{eqnarray*}
Therefore, the result will follow if we obtain $\sup_{\epsilon\in (0,1]}\sup_{t\in[0,T]}\|\nabla u, \nabla\theta\|_{L^{\infty}}< \infty$ for any $T>0$. Like Proposition \ref{5.2}, we need to estimate $\sup_{\epsilon\in (0,1]}\sup_{t\in[0,T]}\|w\|_{L^{p}}^{p}<\infty$ for $p>2$ and $\sup_{\epsilon\in (0,1]}\sup_{t\in[0,T]}(\|\nabla w\|_{L^{4}}^{4}+\|\nabla\theta\|_{L^{4}}^{4})<\infty$, where $w:= \nabla^{\bot}\cdot u$. There is only one additional term, $(\nabla\nabla^{\bot}\cdot fh, \nabla w|\nabla w|^{2})$, and is bounded as follows,
\begin{eqnarray*}
&&|(\nabla\nabla^{\bot}\cdot fh, \nabla w|\nabla w|^{2})|\leq \|\nabla w\|^{3}_{L^{4}}\|\nabla\nabla^{\bot}\cdot f\|_{\mathbb{W}^{0,4}}\|h\|_{H_{0}}\\ &&~\qquad\qquad\qquad\qquad\qquad\leq  C(\|\nabla w\|^{4}_{L^{4}}\|h\|_{H_{0}}+\|\nabla\nabla^{\bot}\cdot f\|_{\mathbb{W}^{0,4}}^{4}\|h\|_{H_{0}}).
\end{eqnarray*}
We take $Y=1+\|\nabla w\|_{L^{4}}^{4}+\|\nabla\theta\|_{L^{4}}^{4}$, $g(t)=1+\|u\|_{L^{\infty}}^{2}+\|w\|_{L^{2}}+\|\nabla w\|_{L^{2}}+\|h\|_{H_{0}}$, $\sigma=(1+\|\nabla\nabla^{\bot}\cdot f\|_{\mathbb{W}^{0,4}}+\|\nabla\nabla^{\bot}\cdot f\|_{\mathbb{W}^{0,4}}\|h\|_{H_{0}}^{\frac{1}{4}})^{4}$ and
\begin{equation*}
\Gamma_{R}=\inf\left\{t\geq 0;\|w\|_{L^{2}}+\|w\|_{L^{4}}+\int_{0}^{t}\|\nabla w\|_{L^{2}}dr+\int_{0}^{t}\|h\|_{H_{0}}dr>R\right\}.
\end{equation*}
By one more application of Lemma \ref{lem5.1}, we have $\sup_{\epsilon\in (0,1]}\sup_{t\in[0,T]}\|\nabla u, \nabla\theta\|_{L^{\infty}}< \infty$ for any $T>0$.

Next, we give the well-posedness result to deterministic controlled equation.
\begin{lemma}\label{lem6.3} Let the initial data $(u_{0},\theta_{0}) \in X^{s}\times H^{s}$ with integer $s>2$, $h\in \mathcal{A}_{M}$. Then, for any $T>0$, $ U_{h}^{0}\in [\mathcal{C}([0,T];X^{s})\cap L^{2}(0,T;X^{s+1})]\times \mathcal{C}([0,T];H^{s})$ is the unique global solution of
\begin{eqnarray}\label{6.5}
\left\{\begin{array}{ll}
du+Audt+P(u\cdot\nabla)udt=P\theta e_{2}dt+Pf hdt,\\
d\theta+(u\cdot\nabla)\theta dt=0,
\end{array}\right.
\end{eqnarray}
with the initial data $U_{0}=(u_{0},\theta_{0})$.
\end{lemma}
\begin{proof} The proof is actually easier that the corresponding stochastic controlled system \eqref{6.3}, we omit the details.
\end{proof}

Let $\mathcal{D}=\left\{\int_{0}^{\cdot}h(t)dt:h\in L^{2}(0,T;H_{0})\right\}\subset \mathcal{C}([0,T];H_{0})$ and define the measurable map $\mathcal{G}^{0}: \mathcal{C}([0,T];H_{0})\rightarrow \mathcal{X}$ by $\mathcal{G}^{0}(g)=U_{h}^{0}$, where $g=\int_{0}^{\cdot}h(t)dt\in \mathcal{D}$ and $U_{h}^{0}$ is the solution to system (\ref{6.5}) and $\mathcal{G}^{0}(g)=0$ otherwise. Let $U_{h_{\epsilon}}^{\epsilon}$ be the solution of system (\ref{6.3}) with $h_{\epsilon}$ in place of $h$. Since pathwise uniqueness of strong solution holds in space $\mathcal{X}$, the Borel-measurable function $\mathcal{G}^{\epsilon}$ satisfies $\mathcal{G}^{\epsilon}\left(\sqrt{\epsilon}\mathcal{W}+\int_{0}^{\cdot}h_{\epsilon}(t)dt\right)=U_{h_{\epsilon}}^{\epsilon}$. Next we establish the weak convergence of the family $\{U^{\epsilon}_{h_{\epsilon}}\}_{\epsilon\in(0,1]}$ as $\epsilon \rightarrow 0$. Its proof uses similar ideas as in the proof of Proposition 5.3 in \cite{Millet}.
\begin{proposition}\label{pro6.1} For every $M<\infty$, let $h_{\epsilon}$ converge to $h$ in distribution as random elements taking values in $\mathcal{A}_{M}$. Then, as $\epsilon\rightarrow 0$ the solution $U_{h_{\epsilon}}^{\epsilon}$ of system (\ref{6.3}) converges in distribution in $\mathcal{X}$ to the solution $U_{h}^{0}$ of system (\ref{6.5}).
\end{proposition}

\begin{proof} To acquire the tightness of the probability measures, we can show that for $\alpha\in [0,\frac{1}{2})$,
\begin{eqnarray}\label{6.7*}
&&\mathbb{E}\|\theta^{\epsilon}\|^2_{W^{1,2}(0,T;H^{s-1})}\leq C,\quad \mathbb{E}\left\|\sqrt{\epsilon}\int_{0}^{t}f d\mathcal{W}\right\|^p_{C^{\alpha}([0,T];H^{s-1})}\leq C,\nonumber\\
 &&\text{and } \mathbb{E}\left\|u^{\epsilon}-\sqrt{\epsilon}\int_{0}^{t}f d\mathcal{W}\right\|^2_{W^{1,2}(0,T;H^{s-1})}\leq C,
\end{eqnarray}
where $C=C(p,M,T,s,\mathbb{T}^{2})$ is a constant independent of $\epsilon$. The eatimates can be achievable using the same argument as Lemma \ref{lem3.1} and condition (\ref{2.13}).

 Also note that $H_{0}\hookrightarrow \hookrightarrow H$ implies $W^{1,2}(0,T;H_{0}) \hookrightarrow \hookrightarrow \mathcal{C}([0,T];H)$ and $h_{\epsilon} \rightarrow h$ in distribution in $L^{2}(0,T;H_{0})$ with the weak topology implies $\int_{0}^{\cdot} h_{\epsilon}dt \rightarrow \int_{0}^{\cdot} hdt$ in distribution in $W^{1,2}(0,T;H_{0})$ with the weak topology denoted by $W^{1,2}(0,T;H_{0})_{w}$. Furthermore, by Lemma \ref{2.3}, we have
\begin{eqnarray}\label{6.13}
\left\{\begin{array}{ll}
\mathcal{C}([0,T];X^{s})\cap \mathcal{C}^\alpha([0,T];X^{s-1})\hookrightarrow\hookrightarrow \mathcal{C}([0,T];X^{s-1}), \\
W^{\alpha,2}(0,T;H^{s-1})\cap L^{2}(0,T;H^{s+1})\hookrightarrow\hookrightarrow L^{2}(0,T;H^{s}),~~{\rm for~ any}~ \alpha\in(0,1).
\end{array}\right.
\end{eqnarray}
Now using the estimates (\ref{6.7*}) and compact embedding \eqref{6.13}, a similar argument as in the proof of Lemma \ref{lem3.2} may be implemented to show the tightness of the family of distributions, $\{\mu^{\epsilon}(\int_{0}^{\cdot}h_{\epsilon}dt, u^{\epsilon},\theta^{\epsilon})\}_{\epsilon\in (0,1]}$ in
\begin{equation*}
\bar{\mathcal{X}}=[W^{1,2}(0,T;H_{0})_{w}\cap \mathcal{C}([0,T];H)]\times [\mathcal{C}([0,T];X^{s-1})\cap L^{2}(0,T;X^{s})]\times \mathcal{C}([0,T];H^{s-1}).
\end{equation*}
Thus, if $\{\epsilon_{n}\}_{n\geq 1}$ is such that $\epsilon_{n} \rightarrow 0$ then for every sequence $(\int_{0}^{\cdot}h_{\epsilon_{n}}dt, u^{\epsilon_{n}}, \theta^{\epsilon_{n}})$, there is a subsequence which we still denote as $(\int_{0}^{\cdot}h_{\epsilon_{n}}dt, u^{\epsilon_{n}}, \theta^{\epsilon_{n}})$ that converges in distribution to $(\int_{0}^{\cdot}hdt, u, \theta)$ in $\bar{\mathcal{X}}$ as $n$ approaches infinity. It remains to confirm that $(u,\theta)$ is the solution of system \eqref{6.5}. For better presentation, we denote $F^{\epsilon}(t):= \int_{0}^{t} h_{\epsilon}dr$ and $F(t):= \int_{0}^{t}hdr$.

Note that $\bar{\mathcal{X}}$ is a quasi-Polish space and hence the Jakubowski-Skorohod representation theorem is invoked to obtain a stochastic basis $(\tilde{\Omega}, \tilde{\mathcal{F}}, \tilde{\mathbb{P}})$ and $\bar{\mathcal{X}}$-valued random variables $(\tilde{F}, \tilde{u},\tilde{\theta})$, and $(\tilde{F}^{\epsilon_{n}}, \tilde{u}^{\epsilon_{n}},\tilde{\theta}^{\epsilon_{n}})$ such that in $\bar{\mathcal{X}}$, $(\tilde{F}, \tilde{u},\tilde{\theta})$ has the same distribution as $(F,u,\theta)$ and $(\tilde{F}^{\epsilon_{n}}, \tilde{u}^{\epsilon_{n}},\tilde{\theta}^{\epsilon_{n}})$ has the same distribution as $(F^{\epsilon_{n}}, u^{\epsilon_{n}}, \theta^{\epsilon_{n}})$, and $(\tilde{F}^{\epsilon_{n}}, \tilde{u}^{\epsilon_{n}},\tilde{\theta}^{\epsilon_{n}}) \rightarrow (\tilde{F}, \tilde{u},\tilde{\theta})$ $\tilde{\mathbb{P}}$ a.s. Therefore, the sequence $(\tilde{u}^{\epsilon_{n}},\tilde{\theta}^{\epsilon_{n}})$ shares the same estimates with $(u^{\epsilon_n}, \theta^{\epsilon_n})$. Namely, there exists a constant $C$  such that
\begin{equation}\label{6.14}
\tilde{\mathbb{E}}\left(\sup_{t\in[0,T]}\|\tilde{u}^{\epsilon_{n}},\tilde{\theta}^{\epsilon_{n}}\|_{s}^{2}\right)+\tilde{\mathbb{E}}\int_{0}^{T}\|\tilde{u}^{\epsilon_{n}}\|_{s+1}^2dt\leq C.
\end{equation}
By the same argument as in subsection 3.2, we obtain
\begin{eqnarray*}
\tilde{u}\in L^{2}(\tilde{\Omega};\mathcal{C}([0,T];X^{s})\cap L^{2}(0,T;X^{s+1})),
\end{eqnarray*}
and
\begin{eqnarray*}
\tilde{\theta}\in L^{2}(\tilde{\Omega};\mathcal{C}([0,T];H^{s})).
\end{eqnarray*}
We next prove that $(\tilde{u},\tilde{\theta})$ is a solution of the following system,
\begin{eqnarray}\label{6.15}
\left\{\begin{array}{ll}
d\tilde{u}+A\tilde{u}dt+P(\tilde{u}\cdot\nabla )\tilde{u}dt=P\tilde{\theta} e_{2}dt+Pf \tilde{h}dt,\\
d\tilde{\theta}+(\tilde{u}\cdot\nabla)\tilde{\theta}dt=0.
\end{array}\right.
\end{eqnarray}
Observe that for any $\phi=(\phi_1,\phi_2)\in L_{{\rm div}}^{2}(\mathbb{T}^{2})\times L^{2}(\mathbb{T}^{2})$,
\begin{align*}
&\quad\left(\tilde{u}^{\epsilon_{n}}(t)-\int_0^t-A\tilde{u}-P(\tilde u\cdot\nabla )\tilde u+P\tilde \theta e_2+Pf\tilde h dr,\phi_1\right)\\&= -\int_{0}^{t} (A\tilde{u}^{\epsilon_{n}}-A\tilde{u},\phi_1)dr - \int_{0}^{t} ((\tilde{u}^{\epsilon_{n}}\cdot \nabla)\tilde{u}^{\epsilon_{n}} - (\tilde{u}\cdot \nabla)\tilde{u}, \phi_1)dr\\
&\quad+ \int_{0}^{t} (\tilde{\theta}^{\epsilon_{n}}e_{2}-\tilde{\theta}e_{2}, \phi_1)dr + \int_{0}^{t} (f(\tilde{h}_{\epsilon_{n}}-\tilde{h}),\phi_1)dr + \sqrt{\epsilon_n}\int_{0}^{t} (f,\phi_1)d\mathcal{W}\\
&= J_{1}+J_{2}+J_{3}+J_{4}+J_{5},\\
&\quad\left(\tilde{\theta}^{\epsilon_{n}}(t)-\int_0^t(\tilde{u}\cdot \nabla)\tilde{\theta}dr,\phi_2\right)=\int_{0}^{t} ((\tilde{u}^{\epsilon_{n}} \cdot \nabla)\tilde{\theta}^{\epsilon_{n}}-(\tilde{u}\cdot \nabla)\tilde{\theta}, \phi_2)dr=I.
\end{align*}
For $J_{1}$, observe that,
\begin{equation*}
\tilde{\mathbb{E}}|J_{1}|\leq \tilde{\mathbb{E}}\int_{0}^{t}\|A\tilde{u}^{\epsilon_{n}}-A\tilde{u}\|_{L^{2}}\|\phi_1\|_{L^{2}}dr\leq \sqrt{T}\|\phi_1\|_{L^{2}}\tilde{\mathbb{E}}\left(\int_{0}^{t}\|\tilde{u}^{\epsilon_{n}}-\tilde{u}\|_{s}^{2}dr\right)^{\frac{1}{2}}.
\end{equation*}
For both $J_{2}$ and $I$, we apply the Sobolev embedding and the Cauchy-Schwarz inequality as follows,
\begin{align*}
\tilde{\mathbb{E}}|J_{2}|&\leq \tilde{\mathbb{E}}\left|\int_{0}^{t}(((\tilde{u}^{\epsilon_{n}}-\tilde{u})\cdot \nabla) \tilde{u}^{\epsilon_{n}}+(\tilde{u}\cdot\nabla)(\tilde{u}^{\epsilon_{n}}-\tilde{u}), \phi_1)dr\right|\nonumber \\
&\leq \|\phi_1\|_{L^{2}}\tilde{\mathbb{E}}\int_{0}^{t}(\|\tilde{u}^{\epsilon_{n}}\|_{s}+\|\tilde{u}\|_{s})\|\tilde{u}^{\epsilon_{n}}
-\tilde{u}\|_{s}dr\nonumber\\
&\leq\sqrt{T}\|\phi_1\|_{L^{2}}\left(\tilde{\mathbb{E}}\sup_{t\in [0,T]}(\|\tilde{u}^{\epsilon_{n}}\|_{s}^{2}
+\|\tilde{u}\|_{s}^{2})\right)^{\frac{1}{2}}\left(\tilde{\mathbb{E}}\int_{0}^{T}\|\tilde{u}^{\epsilon_{n}}-\tilde{u}\|_{s}^{2}dt
\right)^{\frac{1}{2}},
\end{align*}
and
\begin{eqnarray*}
&&\tilde{\mathbb{E}}|I|\leq \tilde{\mathbb{E}}\left|\int_{0}^{t}(((\tilde{u}^{\epsilon_{n}}-\tilde{u})\cdot \nabla) \tilde{\theta}^{\epsilon_{n}}+(\tilde{u}\cdot\nabla)(\tilde{\theta}^{\epsilon_{n}}-\tilde{\theta}), \phi_2)dr\right|\nonumber \\
&&\qquad\leq \|\phi_2\|_{L^{2}}\tilde{\mathbb{E}}\int_{0}^{t}\|\tilde{\theta}^{\epsilon_{n}}\|_{s}\|\tilde{u}^{\epsilon_{n}}-\tilde{u}\|_{s}
+\|\tilde{u}\|_{L^{\infty}}\|\tilde{\theta}^{{\epsilon}_{n}}-\tilde{\theta}\|_{s-2}dr\nonumber\\
&&\qquad\leq T\|\phi_2\|_{L^{2}}\left(\tilde{\mathbb{E}}\sup_{t\in [0,T]}
\|(\tilde{\theta}^{\epsilon_{n}},\tilde{u})\|_{s}^{2}\right)^{\frac{1}{2}}\nonumber\\ &&\qquad\quad\times
\left[\left(\tilde{\mathbb{E}}\int_{0}^{T}\|\tilde{u}^{\epsilon_{n}}-\tilde{u}\|_{s}^{2}dt\right)^{\frac{1}{2}}
+\tilde{\mathbb{E}}\left(\sup_{t\in[0,T]}
\|\tilde{\theta}^{{\epsilon}_{n}}-\tilde{\theta}\|_{s-2}^{2}\right)^{\frac{1}{2}}\right].
\end{eqnarray*}
By assumption, $\tilde{h}_{\epsilon_{n}}\rightarrow\tilde{h}$ as $n\rightarrow \infty$ in $L^{2}(0,T;H_{0})_{w}$ $\tilde{\mathbb{P}}$ a.s. Hence, $\int_{0}^{t}(\tilde{h}_{\epsilon_{n}}-\tilde{h},f^{\ast}\phi)dr\rightarrow 0$ as $n\rightarrow \infty$ and the dominated convergence theorem implies $\tilde{\mathbb{E}}|J_{4}|\rightarrow 0$ as $n\rightarrow \infty$.\
The It\^{o} isometry, the Cauchy-Schwarz inequality, and condition (\ref{2.13}) yield, as $n\rightarrow\infty$
\begin{equation*}
\tilde{\mathbb{E}}|J_{5}|\leq \sqrt{\epsilon_n}\tilde{\mathbb{E}}\left(\int_{0}^{t}
\|f\|_{L_{Q}(H_{0},H)}^{2}\|\phi_1\|_{L^{2}}^{2}dr\right)^{\frac{1}{2}}\leq \sqrt{\epsilon_n}C(T)\|\phi_1\|_{L^{2}}\rightarrow 0.
\end{equation*}
Notice that by bound (\ref{6.14}), $\int_{0}^{T}\|\tilde{u}^{\epsilon_{n}}-\tilde{u}\|_{s}^{2}dt$ and $\sup_{t\in[0,T]}\|\tilde{\theta}^{{\epsilon}_{n}}-\tilde{\theta}\|_{s-2}^{2}$ are bounded in $L^{2}(\tilde{\Omega})$, and hence are uniformly integrable. Therefore, by the Vitali convergence theorem, we have $\tilde{\mathbb{E}}|J_{i}|\rightarrow 0$ and $\tilde{\mathbb{E}}|I|\rightarrow 0$ as $n\rightarrow \infty$ for $i=1,2,3$.

Combining all these estimates, we get as $n\rightarrow \infty$,
\begin{eqnarray*}
\tilde{\mathbb{E}}\left(\tilde{u}^{\epsilon_{n}}(t)-\int_0^t-A\tilde{u}-(\tilde u\cdot\nabla )\tilde u+P\tilde \theta e_2+Pf\tilde h dr, \phi_1\right)\rightarrow 0,
\end{eqnarray*}
and
\begin{eqnarray*}
\tilde{\mathbb{E}}\left(\tilde{\theta}^{\epsilon_{n}}(t)-\int_0^t(\tilde{u}\cdot \nabla)\tilde{\theta}dr,\phi_2\right)\rightarrow 0.
\end{eqnarray*}

In addition, by bound (\ref{6.14}), the Banach-Alaoglu theorem implies,
\begin{eqnarray*}
\tilde{\mathbb{E}}\left[\sup_{t\in [0,T]}((\tilde{u}^{\epsilon_{n}}-\tilde{u}, \tilde{\theta}^{\epsilon_{n}}-\tilde\theta), \phi)\right]\rightarrow 0,~ {\rm as}~ n\rightarrow \infty.
\end{eqnarray*}
Then, we infer that $(\tilde{u},\tilde{\theta})$ is a solution of system (\ref{6.15}), $\tilde{\mathbb{P}}$ a.s. and due to the uniqueness of solutions, $(\tilde{u}, \tilde{\theta})= \tilde{U}_{h}^{0}$, $\tilde{\mathbb{P}}$ a.s. Now $(\tilde{u},\tilde{\theta})$ and $(u,\theta)$ having the same distribution in $\bar{\mathcal{X}}$ implies that $(u,\theta)$ is the solution of system (\ref{6.5}). By Lemma \ref{lem6.3}, we have $(u,\theta)\in [\mathcal{C}([0,T];X^{s})\cap L^{2}(0,T;X^{s+1})]\times \mathcal{C}([0,T];H^{s})$. Thus, for any sequence $(u^{\epsilon_{n}}, \theta^{\epsilon_{n}})$ we may extract a subsequence that converges to $(u,\theta)=U_{h}^{0}$ in distribution in $\mathcal{X}$. This implies that the family $(u^{\epsilon}, \theta^{\epsilon})$ converges to $(u,\theta)=U_{h}^{0}$ in distribution in $\mathcal{X}$ by the sub-subsequence argument.
\end{proof}

The following compactness result is another important factor which allows us to establish the large deviation principle for $U^{\epsilon}$.
\begin{proposition}\label{pro6.2} For every $M<\infty$, let $K_{M}=\{U_{h}^{0}:h\in S_{M}\}$ where $U_{h}^{0}$ is the unique solution in $\mathcal{X}$ of the system (\ref{6.5}). Then, $K_{M}$ is a compact set of $\mathcal{X}$.
\end{proposition}
\begin{proof}By Lemma \ref{lem6.3}, we have $K_{M}\subset \mathcal{X}$. Let $(u^{n},\theta^{n})$ be a sequence of solutions of system \eqref{6.5} in $K_{M}$ corresponding to controls $\{h_{n}\}_{n\geq 1}$ in $S_{M}$ given as follows,
\begin{eqnarray}\label{6.16}
\left\{\begin{array}{ll}
du^{n}+ Au^{n}dt+P(u^{n}\cdot\nabla)u^{n} dt=P\theta^{n} e_{2}dt+Pf h_{n}dt,\\
d\theta^{n}+(u^{n}\cdot\nabla) \theta^{n} dt=0.
\end{array}\right.
\end{eqnarray}
Since $S_{M}$ is a closed and bounded subset of $L^{2}(0,T;H_{0})$, then $\{h_{n}\}_{n\geq 1}$ has a subsequence, which we still denote as $\{h_{n}\}_{n\geq 1}$, that converges weakly to an element $h\in S_{M}$. Similar estimates as \eqref{6.7*} and the bound
\begin{align*}
\bigg\|\int_{0}^{t}Pf h_{n}dr\bigg\|_{W^{1,2}(0,T;H^{s-1})}^{2}&\leq C(T)\int_{0}^{T}\|f h_{n}\|_{s-1}^{2}dt\\
&\leq C(T)\int_{0}^{T}\|f\|_{L_{Q}(H_{0};H^{s-1})}^{2}\|h_{n}\|_{H_{0}}^{2}dt\leq C(T,M),
\end{align*}
imply that $(u^{n},\theta^{n})$ is bounded in $W^{1,2}(0,T;H^{s-1})$ and thus by the compact embedding \eqref{6.13}, there exists a subsequence still denoted by $(u^{n},\theta^{n})$, which converges in $\mathcal{X}$ to some element $(u,\theta)$. By a similar reasoning as in the proof of Proposition \ref{pro6.2}, it can be verified that $(u,\theta)$ is a solution to system \eqref{6.5}.
\end{proof}
With Propositions \ref{pro6.1} and \ref{pro6.2} established, Theorem \ref{the2.3} follows.

\bigskip

\section*{Acknowledgments}
Z. Qiu's research was supported by the CSC under grant No.201806160015. Y. Tang's research was supported in part by the National Natural Science Foundation under grants No.11471129. 

\end{document}